\documentclass[11pt]{article}

\usepackage{hyperref}
\usepackage{amssymb}
\usepackage{amsmath}
\usepackage{amsthm}
\usepackage{amssymb}
\usepackage{algorithm,algpseudocode}
\usepackage{tikz}
\usetikzlibrary{arrows}
\usepackage{flexisym}
\usepackage{multirow}
\usepackage{hhline}
\usepackage{authblk}

\newtheorem{theorem}{Theorem}
\newtheorem{remark}{Remark}
\newtheorem{corollary}{Corollary}
\newtheorem{proposition}{Proposition}

\newcommand{\ve}[1]{\mathbf{#1}}

\DeclareMathOperator*{\argmax}{arg\,max}

\title{Greed Works: An Improved Analysis of Sampling Kaczmarz-Motzkin}

\author[1]{Jamie Haddock}
\author[2]{ Anna Ma}

\affil[1]{Department of Mathematics, University of California, Los Angeles}
\affil[2]{Department of Mathematics, University of California, Irvine}

\begin{document}

\maketitle

\begin{abstract}
Stochastic iterative algorithms have gained recent interest in machine learning and signal processing for solving large-scale systems of equations, $A\ve{x}=\ve{b}$. One such example is the Randomized Kaczmarz (RK) algorithm, which acts only on single rows of the matrix $A$ at a time. While RK randomly selects a row of $A$ to work with, Motzkin's Method (MM) employs a greedy row selection. Connections between the two algorithms resulted in the Sampling Kaczmarz-Motzkin (SKM) algorithm which samples a random subset of $\beta$ rows of $A$ and then greedily selects the best row of the subset. 
Despite their variable computational costs, all three algorithms have been proven to have the same theoretical upper bound on the convergence rate. In this work, an improved analysis of the range of random (RK) to greedy (MM) methods is presented. This analysis improves upon previous known convergence bounds for SKM, capturing the benefit of partially greedy selection schemes. This work also further generalizes previous known results, removing the theoretical assumptions that $\beta$ must be fixed at every iteration and that $A$ must have normalized rows.
\end{abstract}

\section{Introduction}

Large-scale systems of equations arise in many areas of data science, including in machine learning and as subroutines of several optimization methods \cite{boyd2004convex}.  We consider solving these large systems of linear equations, $A \ve{x} = \ve{b}$, where $A \in \mathbb{R}^{m \times n}$, $\ve{b} \in \mathbb{R}^m$, and $m \gg n$.  Iterative methods which use a small portion of the data in each iteration are typically employed in this domain.  These methods offer a small memory footprint and good convergence guarantees.  The Kaczmarz method \cite{Kac37:Angenaeherte-Aufloesung} is such an iterative method that consists of sequential orthogonal projections towards the solution set of a single equation (or subsystem).  Given the system $A\ve{x} = \ve{b}$, the method computes iterates by projecting onto the hyperplane defined by the equation $\ve{a}_i^T \ve{x} = b_i$ where $\ve{a}_i^T$ is a selected row of the matrix $A$ and $b_i$ is the corresponding entry of $\ve{b}$. The iterates are recursively
defined as  $$\ve{x}_{j+1} = \ve{x}_j + \frac{b_i - \ve{a}_{i_{j}}^T \ve{x}_i}{\|\ve{a}_{i_{j}}\|^2} \ve{a}_i.$$ We assume that $A\ve{x} = \ve{b}$ is consistent and $m > n$, but make no assumption on $\text{rank}(A)$.  We will use $\ve{r}_j := A\ve{x}_j - \ve{b}$ to represent the $j$th residual and $\ve{e}_j := \ve{x}_j - \ve{x}^*$ to represent the $j$th error term. We let $A^\dagger$ denote the Moore-Penrose pseudoinverse of the matrix $A$. Additionally, we let $\sigma_{\min}(A)$ be the smallest nonzero singular value of $A$ and unless otherwise noted, we let $\|\cdot\|$ represent the Euclidean norm. We let $\|\cdot\|_F$ denote the Frobenius norm and $\|\cdot\|_\infty$ denote the $\ell^\infty$ norm. A visualization of several iterations of a Kaczmarz method are shown in Figure \ref{fig:Kacz}.  

\begin{figure}
\centering
	\definecolor{uuuuuu}{rgb}{0.26666666666666666,0.26666666666666666,0.26666666666666666}
	\definecolor{ududff}{rgb}{0.30196078431372547,0.30196078431372547,1.}
	\begin{tikzpicture}[line cap=round,line join=round,>=triangle 45,x=1.0cm,y=1.0cm]
	\clip(-2.7258834270895766,-0.5869710300152282) rectangle (5.472915391462454,5.919351551881352);
	\draw [line width=1.pt,domain=-2.7258834270895766:5.472915391462454] plot(\x,{(--6.9336-2.7*\x)/2.34});
	\draw [line width=1.pt,domain=-2.7258834270895766:5.472915391462454] plot(\x,{(-0.684-3.36*\x)/-2.44});
	\draw [line width=1.pt,domain=-2.7258834270895766:5.472915391462454] plot(\x,{(-8.2868--1.58*\x)/-3.8});
	\draw [line width=1.pt,domain=-2.7258834270895766:5.472915391462454] plot(\x,{(--6.3964--0.86*\x)/4.2});
	\draw [line width=1.pt,domain=-2.7258834270895766:5.472915391462454] plot(\x,{(--1.8328--2.9*\x)/2.82});
	\draw [line width=1.pt,domain=-2.7258834270895766:5.472915391462454] plot(\x,{(--0.0168--3.3*\x)/2.02});
	\draw [line width=1.pt,domain=-2.7258834270895766:5.472915391462454] plot(\x,{(--5.3888-2.72*\x)/1.44});
	\draw (1.8243805534266095,5.741896431118171) node[anchor=north west] {$\ve{x}_0$};
	\draw [line width=1.pt,dash pattern=on 2pt off 2pt] (1.82,5.48)-- (2.932491316197297,4.799020467054989);
	\draw [line width=1.pt,dash pattern=on 2pt off 2pt] (2.932491316197297,4.799020467054989)-- (3.4989740699715863,4.248164823729645);
	\draw [line width=1.pt,dash pattern=on 2pt off 2pt] (0.5567775179380073,2.690531355005986)-- (3.4989740699715863,4.248164823729645);
	\draw (2.934201036479338,5.062131385248379) node[anchor=north west] {$\ve{x}_1$};
	\draw (3.502984034043861,4.421093899760177) node[anchor=north west] {$\ve{x}_2$};
	\draw (0.5619597539541311,2.8534724674482086) node[anchor=north west] {$\ve{x}_3$};
	\begin{scriptsize}
	\draw [fill=ududff] (1.06,1.74) circle (2.5pt);
	\draw [fill=ududff] (1.82,5.48) circle (2.5pt);
	\draw [fill=uuuuuu] (2.932491316197297,4.799020467054989) circle (2.0pt);
	\draw [fill=uuuuuu] (3.4989740699715863,4.248164823729645) circle (2.0pt);
	\draw [fill=uuuuuu] (0.5567775179380073,2.690531355005986) circle (2.0pt);
	\end{scriptsize}
	\end{tikzpicture}
	\caption{Several iterations of a Kaczmarz method. The iterate $\ve{x}_{j+1}$ is the orthogonal projection of $\ve{x}_{j}$ onto the solution set of the selected equation (represented by a line).}\label{fig:Kacz}
\end{figure}
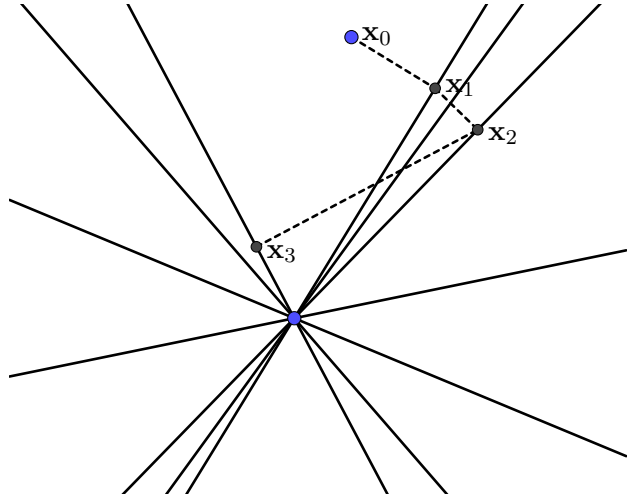

The Kaczmarz method was originally proposed in the late 30s~\cite{Kac37:Angenaeherte-Aufloesung} and rediscovered in the 1970's under the name \emph{algebraic reconstruction technique (ART)} as an iterative method for reconstructing an image from a series of angular projections in computed tomography \cite{GBH70:Algebraic-Reconstruction, HM93:Algebraic-Reconstruction}. This method has seen popularity among practitioners and researchers alike since the beginning of the digital age \cite{censor1983strong, hanke1990acceleration}, but saw a renewed surge of interest after the elegant convergence analysis of the \emph{Randomized Kaczmarz (RK) method} in \cite{SV09:Randomized-Kaczmarz}.  In \cite{SV09:Randomized-Kaczmarz}, the authors showed that for a consistent system with unique solution, RK (with specified sampling distribution) converges at least linearly in expectation with the guarantee
\begin{equation}
    \mathbb{E}\|\ve{e}_k\|^2 \le \left(1 - \frac{\sigma_{\min}^2(A)}{\|A\|_F^2}\right)^k \|\ve{e}_0\|^2. \label{eq:RKrate}
\end{equation}
Many variants and extensions followed, including convergence analyses for inconsistent and random linear systems \cite{Nee10:Randomized-Kaczmarz, CP12:Almost-Sure-Convergence}, connections to other popular iterative algorithms \cite{Ma2015convergence, NSWjournal, Pop01:Fast-Kaczmarz-Kovarik, Pop04:Kaczmarz-Kovarik-Algorithm, frek}, block approaches \cite{needell2013paved, popa2012kaczmarz}, acceleration and parallelization strategies \cite{EN11:Acceleration-Randomized, liu2014asynchronous, morshed2019accelerated, moorman2020randomized}, and techniques for reducing noise and corruption \cite{ZF12:Randomized-Extended, HN18Corrupted}.

Another popular Kaczmarz method extension is greedy (rather than randomized) row selection, which has been rediscovered several times in the literature as the ``most violated constraint control" or the ``maximal-residual control" \cite{CensorRowAction, GreedyKaczmarz, Petra2016}.  This method was proposed in the 1950's as an iterative relaxation method for linear programming by Agmon, Motzkin, and Schoenberg under the name \emph{Motzkin's relaxation method for linear inequalities (MM)} \cite{motzkinschoenberg, agmon}.  In \cite{agmon}, the author showed that MM converges at least linearly (deterministically) with the convergence rate of \eqref{eq:RKrate}. The bodies of literature studying this greedy strategy have remained somewhat disjoint, with analyses for linear systems of equations in the numerical linear algebra community and analyses for linear systems of inequalities in the operations research and linear programming community \cite{goffin, goffinnonpoly, telgen, amaldihauser, betkelp, BetkeGritzmann, chubanovlp}.  There has been recent work in analyzing variants of this greedy strategy \cite{du2019new, bai2018greedy, bai2018relaxed, RN19sketchasilo}.  In \cite{RN19sketchasilo}, the authors analyze MM on a system to which a Gaussian sketch has been applied.  In \cite{bai2018greedy,bai2018relaxed}, the authors analyze variants of MM in which the equation selected in each iteration is chosen randomly amongst the set whose residual values are sufficiently near the maximal residual value.  In \cite{du2019new}, the authors provide a convergence analysis for a generalized version of MM in which the equation chosen in each iteration is that which has the maximal \emph{weighted residual} value which are the residual values divided by the norm of the corresponding row of the measurement matrix.
In \cite{DLHN16SKM}, the authors illustrated the connection between MM and RK and proposed a family of algorithms that interpolate between the two, known as the \emph{Sampling Kaczmarz-Motzkin (SKM) methods}. 

The SKM methods operate by randomly sampling a subset of the system of equations, computing the residual of this subset, and projecting onto the equation corresponding to the largest magnitude entry of this sub-residual. The family of methods (parameterized by the size of the random sample of equations, $\beta$) interpolates between MM, which is SKM with $\beta = m$, and RK, which is SKM with $\beta = 1$.  In \cite{DLHN16SKM}, the authors prove that the SKM methods converge at least linearly in expectation with the convergence rate specified in \eqref{eq:RKrate}. Meanwhile, the empirical convergence of this method is seen to depend upon $\beta$; however, increasing $\beta$ also increases the computational cost of each iteration so the per iteration gain from larger sample size may be outweighed by the in-iteration cost.  This is reminiscent of other methods which use sampled subsets of data in each iteration, such as the block projection methods \cite{AC89:Block-Iterative-Projection, needell2013paved, needell2015randomized}. 

Like SKM, the randomized block Kaczmarz (RBK) methods use a subset of rows $\tau \subset [m]$ to produce the next iterate; rather than forcing the next iterate to satisfy the single sampled equation as in RK, block iterates satisfy \emph{all} the equations in the randomly sampled block.  The $(k+1)$st RBK iteration is given by $$\ve{x}_{k+1} = \ve{x}_k + (A_\tau)^\dagger (\ve{b}_\tau - A_\tau \ve{x}_k),$$ where $A_\tau$ and $\ve{b}_\tau$ represent the restriction onto the row indices in $\tau$.  In \cite{needell2013paved}, the authors prove that on a system with a row-normalized measurement matrix and a well-conditioned row-paving RBK converges at least linearly in expectation with the guarantee
\begin{equation}
    \mathbb{E}\|\ve{e}_k\|^2 \le \left(1 - \frac{\sigma_{\min}^2(A)}{C\|A\|^2 \log(m+1)}\right)^k \|\ve{e}_0\|^2.
\end{equation}
where $C$ is an absolute constant and $\|A\|$ denotes the operator norm of the matrix.  This can be a significant improvement over the convergence rate of \eqref{eq:RKrate} when $\|A\|_F^2 \gg \|A\|^2 \log(m+1)$.  However, the cost per iteration scales with the size of the blocks.
In \cite{needell2015randomized}, the authors generalize this result to inconsistent systems and show that, up to a convergence horizon, RBK converges to the least-squares solution. 

In \cite{gower2015randomized, gower2015stochastic}, the authors introduce a framework of iterative methods known as the \emph{sketch-and-project} methods.  The sketch-and-project framework of methods produce each new iterate by projecting the previous iterate onto a sketch of the linear system; i.e., $\ve{x}_{k+1} = \text{argmin}_{\ve{x} \in \mathbb{R}^n} \|\ve{x} - \ve{x}_k\|_B^2 \text{ s.t. } S_k^\top A \ve{x} = S_k^\top \ve{b}$.  Subsequent works proposed and analyzed variants with momentum \cite{loizou2017momentum}, inexact variants \cite{loizou2019convergence}, and adaptive variants \cite{gower2019adaptive}.  This framework includes as special cases many forms of row- and column-action methods and second-order iterative least-squares methods \cite{xiang2017randomized}.  Kaczmarz methods which iteratively project onto the solution spaces of subsets of rows in each iteration (like Block RK or SKM) can be interpreted and analyzed in this framework.  Single row-action methods are recovered when the sketching matrices select a single row of the system.  The SKM methods are recovered when the sketching matrices select a single row of the system and the choice of which sketch to use in each iteration is made in the same way as SKM (a randomized sample then a greedy selection based upon sketched residual).  The results recovered from \cite{gower2015randomized} for this interpretation of SKM coincide with \eqref{eq:RKrate}.

In \cite{needell2019block}, the authors, inspired by the sketching framework in \cite{gower2015randomized}, construct a block-type method which iterates by projecting onto a Gaussian sketch of the equations.  They show that this method converges at least linearly in expectation with the guarantee
\begin{equation}
    \mathbb{E}\|\ve{e}_k\|^2 \le \left(1 - \left[\frac{\sqrt{s}\sigma_{\min}(A)}{9\sqrt{s}\|A\| + C\|A\|_F}\right]^2\right)^k \|\ve{e}_0\|^2. 
\end{equation}
where $C$ is an absolute constant and $s$ is the number of rows in the resulting sketched system.
This result requires a Gaussian sketch which is a costly operation, however the authors suggest using a Gaussian sketch of only a subset of the equations.  This result is most related to SKM and to our main result due to the presence of $s$, the size of the sketched system, in the bound.

\section{Previous Results}

This section focuses on the convergence behavior of the RK, MM, and SKM methods. Each of these projection methods is a special case of Algorithm \ref{alg:Kacz} with a different \emph{selection rule} (Line~\ref{line:selection}). In iteration $j$, RK uses the randomized selection rule that chooses $t_j = i$ with probability $\|\ve{a}_{i}\|_2^2/\|A\|_F^2$, MM uses the greedy selection rule $t_j = \argmax_i |\ve{a}_i^\top\ve{x}_{j-1} - {b_i}|$, and SKM uses the hybrid selection rule that first samples a subset of $\beta$ rows, $\tau_j$, uniformly at random from all subsets of size $\beta$, $\tau_j \sim \text{unif}(\binom{[m]}{\beta})$, and then chooses $t_j = \argmax_{i \in \tau_j} |\ve{a}_i^\top\ve{x}_{j-1} - b_i|$. As previously mentioned, RK and MM are special cases of the SKM method when the sample size $\beta = 1$ and $\beta = m$, respectively.  Each of the methods converge linearly when the system is consistent with unique solution (RK and SKM converge linearly in expectation, MM converges linearly deterministically).  In Table~\ref{tab:iteratedrates}, we present the selection rules and convergence rates for RK, MM, and SKM.  Note that under the assumption that $A$ has been normalized so that $\|\ve{a}_i\|^2 = 1$, each of these upper bounds on the convergence rate is the same since $\|A\|_F^2 = m$.  Thus, these results do not reveal any advantage the more computationally expensive methods (MM, SKM with $\beta \gg 1$) enjoy over RK.  There are, in fact, pathological examples on which RK, MM, and SKM exhibit nearly the same behavior (e.g., consider the system defining two lines that intersect at one point in $\mathbb{R}^2$), so it is not possible to prove significantly different convergence rates without leveraging additional properties of the system.

\begin{algorithm}
	\caption{Generic Kaczmarz Method}
	\begin{algorithmic}[1]
		\Procedure{Kacz}{$A,\ve{b},\ve{x}_0$}
		\State $k = 1$
		\Repeat
		\State Choose $t_k \in [m]$ according to selection rule. \label{line:selection}
		\State $\ve{x}_k = \ve{x}_{k-1} -  \frac{\ve{a}_{t_k}^T\ve{x}_{k-1} - b_{t_k}}{\|\ve{a}_{t_k}\|_2^2}\ve{a}_{t_k}$. 
		\State $k = k+1$
		\Until{stopping criterion reached}
		\State \textbf{return} $\ve{x}_k$ 
		\EndProcedure
	\end{algorithmic}
	\label{alg:Kacz}
\end{algorithm}

\begin{table}
\begin{center}
\begin{tabular}{| c | c | c |}
\hline
 & Selection Rule & Convergence Rate \\
\hline
RK \cite{SV09:Randomized-Kaczmarz} & $\mathbb{P}(t_j = i) = \frac{\|\ve{a}_i\|^2}{\|A\|_F^2}$ & $\mathbb{E}\|\ve{e}_k\|^2 \le (1 - \frac{\sigma_{\min}^2(A)}{\|A\|_F^2})^k \|\ve{e}_0\|^2$  \\
\hline
SKM \cite{DLHN16SKM} & \begin{tabular}{c} $\tau_j \sim \text{unif}(\binom{[m]}{\beta})$ \\ $t_j = \argmax_{i \in \tau_j} |\ve{a}_i^\top \ve{x}_{j-1} - b_i|$ \end{tabular} & $\mathbb{E}\|\ve{e}_k\|^2 \le (1 - \frac{\sigma_{\min}^2(A)}{m})^k \|\ve{e}_0\|^2$ \\
\hline
MM \cite{agmon} & $t_j = \argmax_i |\ve{a}_i^\top \ve{x}_{j-1} - b_i|$ & $\|\ve{e}_k\|^2 \le (1 - \frac{\sigma_{\min}^2(A)}{m})^k \|\ve{e}_0\|^2$ \\
\hline
\end{tabular}
\end{center}
\caption{The selection rules and convergence rates of RK, SKM, and MM. The presented results for MM and SKM assume that $A$ has been normalized so that $\|\ve{a}_i\|^2 = 1$.}\label{tab:iteratedrates}
\end{table}

In \cite{HN18Motzkin}, the authors demonstrate that MM can converge faster than RK or SKM and that the convergence rate depends on the structure of the residual terms of the iterations, $\ve{r}_k = A\ve{x}_k - \ve{b}$.  In particular, they prove that $$\|\ve{e}_{k}\|^2 \le \Pi_{j=0}^{k-1} \bigg(1 - \frac{\sigma_{\min}^2(A)}{4\gamma_j}\bigg) \|\ve{e}_0\|^2,$$ where $\gamma_j$ is the \emph{dynamic range} of the $i$th residual, $\gamma_j := \frac{\|\ve{r}_j\|^2}{\|\ve{r}_j\|_\infty^2}$.  Our main contribution in this paper is to prove that the SKM methods can exhibit a similarly accelerated convergence rate and the advantage scales with the size of the sample, $\beta$.  Again, this advantage depends upon the structure of the residuals of the iterations.  We define here a generalization of the dynamic range used in \cite{HN18Motzkin}; our dynamic range is defined as 
\begin{equation}
\gamma_j = \frac{\sum_{\tau \in {[m] \choose \beta}} \|A_{\tau}\ve{x}_{j-1} - \ve{b}_{\tau}\|^2}{\sum_{\tau \in {[m] \choose \beta}} \|A_{\tau}\ve{x}_{j-1} - \ve{b}_{\tau}\|_\infty^2}.
\label{eq:gammaj}
\end{equation}  

Now, we let $\mathbb{E}_{\tau_j}$ denote expectation with respect to the random sample $\tau_j$ conditioned upon the sampled $\tau_i$ for $i < j$, and $\mathbb{E}$ denote expectation with respect to all random samples $\tau_{i}$ for $1 \le i \le j$ where $j$ is understood to be the last iteration in the context in which $\mathbb{E}$ is applied. We state our main result below in Corollary \ref{cor:SKMconvergence}; this is a corollary of our generalized result which will be discussed and proven later. 
\begin{corollary}
Let $A$ be normalized so $\|\ve{a}_i\| = 1$ for all rows $i = 1,...,m$. {Suppose the system of equations $A \ve{x} = \ve{b}$ is consistent, define $\ve{x}^* = A^\dagger \ve{b}$, and let $\ve{x}_0 \in \text{range}(A^\top)$.}  Then SKM converges at least linearly in expectation and the bound on the
rate depends on the dynamic range, $\gamma_k$ of the random sample of $\beta$ rows of $A$, $\tau_k$. Precisely, in the $k$th iteration of SKM, we have $$\mathbb{E}_{\tau_k} \|\ve{x}_k - \ve{x}^*\|^2 \le \bigg(1 - \frac{\beta \sigma_{\min}^2(A)}{\gamma_k m}\bigg) \|\ve{x}_{k-1} - \ve{x}^*\|^2,$$ so applying expectation with respect to all iterations, we have 
\begin{equation*}
\mathbb{E}\|\ve{x}_k - \ve{x}^*\|^2 \le \prod_{j=1}^k \bigg(1 - \frac{\beta \sigma_{\min}^2(A)}{\gamma_j m}\bigg) \|\ve{x}_0 - \ve{x}^*\|^2.
\end{equation*} \label{cor:SKMconvergence}
\end{corollary}
Corollary~\ref{cor:SKMconvergence} shows that SKM experiences at least linear convergence where the contraction term is a product of terms that are less than one and dependent on the sub-sample size $\beta$. When $\beta = 1$, as in RK, $\gamma_k$ = 1, so Corollary~\ref{cor:SKMconvergence} recovers the upper bound for RK shown in~\cite{SV09:Randomized-Kaczmarz}. However, when 
$\beta = m$ for MM, Corollary \ref{cor:SKMconvergence} offers an improved upper bound on the error over \cite{HN18Motzkin}; specifically $$\|\ve{e_k}\|^2 \le \bigg(1 - \frac{\sigma_{\min}^2(A)}{\gamma_k}\bigg)\|\ve{e}_{k-1}\|^2.$$

Our result illustrates that the progress made by an iteration of the SKM algorithm depends upon the dynamic range of the residual of that iteration.  The dynamic range of each iteration, $\gamma_j$, satisfies $$1 \le \gamma_j \le \beta.$$  Note that the upper bound, $\gamma_j = \beta$, is achieved by a constant residual where $|\ve{a}_i^T \ve{x}_{j} - b_i| = |\ve{a}_{i^\prime}^T \ve{x}_j - b_{i^\prime}|$ for all $i,i^\prime \in [m]$, while the lower bound is achieved by the residual with one nonzero entry.  As smaller $\gamma_j$ provides a smaller upper bound on the new error $\ve{e}_j$, we consider the situation with one nonzero entry in the residual as the ``best case" and the situation with a constant residual as the ``worst case."  We now compare our single iteration result in the best and worst cases to the previously known single iteration results of \cite{agmon, DLHN16SKM, HN18Motzkin, SV09:Randomized-Kaczmarz}.  These are summarized in Table \ref{tab:iterationcomparison}; we present only the contraction terms $\alpha$ such that $$\mathbb{E}_{\tau_k} \|\ve{e}_k\|^2 \le \alpha \|\ve{e}_{k-1}\|^2,$$ for each upper bound in the case that $A$ is normalized so that $\|\ve{a}_i\|^2 = 1$ for $i \in [m]$.  In particular, note that the worst case residual provides the same upper bound rate as those of \cite{SV09:Randomized-Kaczmarz, DLHN16SKM, agmon}.  

\begin{table}
\begin{center}
\begin{tabular}{| c | c | c | c | c |}
\hline
 & Best Case & Worst Case & Previous Best Case & Previous Worst Case\\
\hline
MM & $1 - \sigma_{\min}^2(A)$ 
& \multirow{3}{*}{$1 - \frac{\sigma_{\min}^2(A)}{m}$} &  $1 - \frac{\sigma_{\min}^2(A)}{4}$ \cite{HN18Motzkin} &\multirow{3}{*}{$1 - \frac{\sigma_{\min}^2(A)}{m}$ \cite{agmon, DLHN16SKM, SV09:Randomized-Kaczmarz}} \\
\hhline{--~-~}
SKM & $ 1 - \frac{\beta\sigma_{\min}^2(A)}{m}$ & & \multirow{2}{*}{$1 - \frac{\sigma_{\min}^2(A)}{m}$}  & \\
\hhline{--~~~}
RK & $ 1 - \frac{\sigma_{\min}^2(A)}{m}$ & & &   \\
\hline
\end{tabular}
\end{center}
\caption{Contraction terms $\alpha$ such that $\mathbb{E}_{\tau_k} \|\ve{e}_k\|^2 \le \alpha \|\ve{e}_{k-1}\|^2$ for the best and worst case bounds of MM, SKM, and RK.}\label{tab:iterationcomparison}
\end{table}

\section{Main Results}
Corollary \ref{cor:SKMconvergence} is a specialization of our general result to SKM with a fixed sample size $\beta$ and systems that are row-normalized.  Our general result requires neither row-normalization nor a static sample size.  However, we must additionally generalize the SKM sampling distribution for systems that are not row-normalized.  We now consider the general SKM method which samples $\beta_k$ many rows of $A$ in the $k$th iteration (according to probability distribution $p_{\ve{x}_{k-1}}$ defined in \eqref{eq:probdist}) and projects onto the hyperplane associated to the largest magnitude entry of the sampled sub-residual. 

The generalized probability distribution over the subset of rows of $A$ of size $\beta_k$ is denoted $p_{\ve{x}}: {[m] \choose \beta_k} \rightarrow [0,1).$  The sampled subset of rows of $A$, $\tau_k \sim p_{\ve{x}}$ where 
\begin{equation}
    p_{\ve{x}}(\tau_k) = \frac{\|\ve{a}_{t(\tau_k,\ve{x})}\|^2}{\sum_{\tau \in {[m] \choose \beta_k}} \|\ve{a}_{t(\tau,\ve{x})}\|^2},
    \label{eq:probdist}
\end{equation}
and $t(\tau,\ve{x}) = \argmax_{t \in \tau} (\ve{a}_{t}^\top\ve{x} - b_t)^2.$  Thus, our generalized SKM method is Algorithm \ref{alg:Kacz} with selection rule $\tau_j \sim p_{\ve{x}_{j-1}}$ and $t_j = t(\tau_j,\ve{x}_{j-1}).$ Similar to the RK probability distribution of \cite{SV09:Randomized-Kaczmarz}, the computation of~\eqref{eq:probdist} is utilized here simply to theoretically analyze the SKM algorithm without requiring normalized rows. This choice of sampling distribution conveniently simplifies the expected value computation in the proof of Theorem \ref{thm:gSKMconvergence} by cancelling the numerator of the probability with the squared norm of the sampled row. We do not suggest that this probability distribution be implemented in a real world setting as it is computationally prohibitive.  

One could instead implement a uniform distribution over rows or learn the distribution with probabilities proportional to the squared norms of the rows (as suggested in \cite{SV09:Randomized-Kaczmarz}).  Neither of these is guaranteed to coincide with the distribution defined in (\ref{eq:probdist}), due to the dependence on the iterate $\ve{x}$. However, for many datasets where the row norms are (approximately) equal, the uniform distribution (approximately) coincides with (\ref{eq:probdist}).  In particular, when the rows of $A$ all have equal norm, as in the case of incidence matrices (see Section \ref{sec:avgcons}), then (\ref{eq:probdist}) reduces to the uniform distribution over samples of size $\beta_k$.  Past works which analyze SKM \cite{DLHN16SKM, morshed2019accelerated, morshed2020generalization} assume that the rows of $A$ are normalized and that the probability distribution over the samples of size $\beta$ is uniform.  To the best of our knowledge, ours is the first work in this area to analyze an iterative projection method with an iteration dependent sampling distribution.

Our main result shows that the generalized SKM converges at least linearly in expectation with a bound that depends on the dynamic range of the sampled sub-residual, the size of the sample, and the minimum squared nonzero singular value of $A$, $\sigma_{\min}^2(A)$. In the event that there are multiple rows within the sub-residual which achieve $\max_{t \in \tau} (\ve{a}_{t}^\top\ve{x} - b_t)^2$, an arbitrary choice can be made amongst those rows and the main result will not be affected by this choice.

Theorem~\ref{thm:gSKMconvergence} provides theoretical convergence guarantees for the generalized SKM method. Whereas previous guarantees for SKM required normalized rows or fixed sample sizes $\beta$~\cite{DLHN16SKM, HN18Motzkin, morshed2019accelerated, morshed2020generalization}, the guarantees presented here do not require either assumption. In addition, the contraction term of the generalized SKM method shows dependence on the dynamic range, another feature lacking in previous works. Following the statement of the theorem, we use standard techniques in the Kaczmarz literature to prove our main result. 

We additionally describe a simple generalization of the main result in the case that the samples are not made according the generalized SKM distribution~\eqref{eq:probdist}, but instead according to a distribution $\tilde{p}(\tau)$ whose probabilities are at least a constant factor of those in \eqref{eq:probdist}. We also remark on the special case in which rows have equal norm (and thus subsets $\tau$ are selected uniformly at random), the case where $\beta_k$ is fixed, and the case in which $\beta_k = 1$ in order to make connections to previous results. Due to the dependence of the sampling distribution upon the current iterate, our main result is not easily iterable to provide the usual form of a Kaczmarz type result (e.g., $\mathbb{E}\|\ve{e}_k\|^2 \le \alpha^k \|\ve{e}_0\|^2$) so we present the bound for only a single iteration.  However, in the special cases we describe in Remarks \ref{rem:equalrows} and \ref{rem:MMoriginal} we are able to iterate the simplified expression due to the simplicity of the sampling distribution.

\begin{theorem}
Suppose the system of equations $A \ve{x} = \ve{b}$ is consistent, define $\ve{x}^* = A^\dagger \ve{b}$, and let $\ve{x}_0 \in \text{range}(A^\top)$.  Then generalized SKM converges at least linearly in expectation and the bound on the
rate depends on the dynamic range, $\gamma_k$ of the random sample of $\beta_k$ rows of $A$, $\tau_k$. Precisely, in the $k$th iteration of generalized SKM, we have 
\begin{equation*}
\mathbb{E}_{\tau_k} \|\ve{x}_k - \ve{x}^*\|^2 \le \bigg(1 - \frac{\beta_k{m \choose \beta_k} \sigma_{\min}^2(A)}{\gamma_k m \sum_{\tau \in {[m] \choose \beta_k}}\|\ve{a}_{t(\tau,\ve{x}_{k-1})}\|^2}\bigg) \|\ve{x}_{k-1} - \ve{x}^*\|^2.
\end{equation*}
\label{thm:gSKMconvergence}
\end{theorem}
\begin{proof} We begin by rewriting the generalized SKM iterate $\ve{x}_k$ and simplifying the resulting expression which yields
\begin{align*}
    \|\ve{x}_k - \ve{x}^*\|^2 &= \Bigg\|\ve{x}_{k-1} - \frac{\ve{a}_{t(\tau_k,\ve{x}_{k-1})}^\top \ve{x}_{k-1} - b_{t(\tau_k,\ve{x}_{k-1})}}{\|\ve{a}_{t(\tau_k,\ve{x}_{k-1})}\|^2} \ve{a}_{t(\tau_k,\ve{x}_{k-1})} - \ve{x}^*\Bigg\|^2
    \\&= \|\ve{x}_{k-1} - \ve{x}^*\|^2 - \frac{(\ve{a}_{t(\tau_k,\ve{x}_{k-1})}^\top \ve{x}_{k-1} - b_{t(\tau_k,\ve{x}_{k-1})})^2}{\|\ve{a}_{t(\tau_k,\ve{x}_{k-1})}\|^2}
    \\&= \|\ve{x}_{k-1} - \ve{x}^*\|^2 - \frac{\|A_{\tau_k} \ve{x}_{k-1} - \ve{b}_{\tau_k}\|_\infty^2}{\|\ve{a}_{t(\tau_k,\ve{x}_{k-1})}\|^2},
\end{align*}
where the first equation uses the definition of the generalized SKM iterate, and the second follows from the fact that $\ve{a}_{t(\tau_k,\ve{x}_{k-1})}^\top (\ve{x}_{k-1} - \ve{x}^*) = \ve{a}_{t(\tau_k,\ve{x}_{k-1})}^\top \ve{x}_{k-1} - b_{t(\tau_k,\ve{x}_{k-1})}$.  Note that $$\ve{a}_{t(\tau_k,\ve{x}_{k-1})}^\top \ve{x}^* = (AA^\dagger\ve{b})_{t(\tau_k,\ve{x}_{k-1})} = b_{t(\tau_k,\ve{x}_{k-1})}$$ since $\ve{b} \in \text{range}(A)$ and $AA^\dagger$ is the projector onto $\text{range}(A)$.

Now, we take expectation of both sides (with respect to the sampled $\tau_k$ according to the distribution \eqref{eq:probdist}).  This gives
\begin{align}
    \mathbb{E}_{\tau_k} \|\ve{x}_k - \ve{x}^*\|^2 &= \|\ve{x}_{k-1} - \ve{x}^*\|^2 - \mathbb{E}_{\tau_k} \frac{\|A_{\tau_k} \ve{x}_{k-1} - \ve{b}_{\tau_k}\|_\infty^2}{\|\ve{a}_{t(\tau_k,\ve{x}_{k-1})}\|^2} \nonumber
    \\&= \|\ve{x}_{k-1} - \ve{x}^*\|^2 - \sum_{\tau \in {[m] \choose \beta_k}} p_{\ve{x}_{k-1}}(\tau) \frac{\|A_{\tau} \ve{x}_{k-1} - \ve{b}_{\tau}\|_\infty^2}{\|\ve{a}_{t(\tau,\ve{x}_{k-1})}\|^2} \nonumber
    \\&=\|\ve{x}_{k-1} - \ve{x}^*\|^2 - \sum_{\tau \in {[m] \choose \beta_k}} \frac{\|\ve{a}_{t(\tau,\ve{x}_{k-1})}\|^2}{\sum_{\pi \in {[m] \choose \beta_k}} \|\ve{a}_{t(\pi,\ve{x}_{k-1})}\|^2} \frac{\|A_{\tau} \ve{x}_{k-1} - \ve{b}_{\tau}\|_\infty^2}{\|\ve{a}_{t(\tau,\ve{x}_{k-1})}\|^2} \label{eq:probchangeremark}
    \\&=\|\ve{x}_{k-1} - \ve{x}^*\|^2 - \frac{1}{\sum_{\pi \in {[m] \choose \beta_k}} \|\ve{a}_{t(\pi,\ve{x}_{k-1})}\|^2}\sum_{\tau \in {[m] \choose \beta_k}} \|A_{\tau} \ve{x}_{k-1} - \ve{b}_{\tau}\|_\infty^2 \nonumber
    \\&= \|\ve{x}_{k-1} - \ve{x}^*\|^2 - \frac{1}{\gamma_k \sum_{\pi \in {[m] \choose \beta_k}} \|\ve{a}_{t(\pi,\ve{x}_{k-1})}\|^2}\sum_{\tau \in {[m] \choose \beta_k}} \|A_{\tau} \ve{x}_{k-1} - \ve{b}_{\tau}\|^2 \nonumber
    \\&= \|\ve{x}_k - \ve{x}^*\|^2 - \frac{{m \choose \beta_k}\beta_k}{\gamma_k m \sum_{\pi \in {[m] \choose \beta_k}} \|\ve{a}_{t(\pi,\ve{x}_{k-1})}\|^2} \|A \ve{x}_{k-1} - \ve{b}\|^2 \nonumber
    \\&\le \Bigg(1 - \frac{{m \choose \beta_k}\beta_k \sigma_{\min}^2(A)}{\gamma_k m \sum_{\pi \in {[m] \choose \beta_k}} \|\ve{a}_{t(\pi,\ve{x}_{k-1})}\|^2}\Bigg)\|\ve{x}_{k-1} - \ve{x}^*\|^2, \nonumber
\end{align}
where the last line follows from standard properties of singular values and the fact that $\ve{x}_{k-1} \in \text{range}(A^\top)$ (since $\ve{x}_0 \in \text{range}(A^\top)$ and  the SKM update preserves membership in $\text{range}(A^\top)$). This completes our proof.
\end{proof}

Now, we provide a corollary of the previous result which provides a bound on the expected error for the SKM algorithm which samples subsets of rows $\tau_k$ according to an alternate probability distribution $\tilde{p}(\tau)$ satisfying $\tilde{p}(\tau) \ge \epsilon p_{\ve{x}_{k-1}}(\tau)$ for all $\tau \in {[m] \choose \beta_k}$.  In this case, we can exploit the relationship between probabilities to reuse the proof of Theorem~\ref{thm:gSKMconvergence}.  Provided that the probability distribution $\tilde{p}(\tau)$ is fixed between iterations we can iterate the bound unlike in Theorem~\ref{thm:gSKMconvergence}.  
An application of Corollary~\ref{cor:alternateprob} with the uniform distribution is given in Remark~\ref{rem:unif}.

\begin{corollary} 
Suppose the system of equations $A \ve{x} = \ve{b}$ is consistent, define $\ve{x}^* = A^\dagger \ve{b}$, and let $\ve{x}_0 \in \text{range}(A^\top)$.
Suppose one runs SKM with $\tau_k \sim \tilde{p}(\tau)$ and $t_k = t(\tau_k,\ve{x}_{k-1})$ and that the probabilities used to sample, $\tilde{p}(\tau)$, are at least a constant factor of the probabilities~\eqref{eq:probdist}; that is $\tilde{p}(\tau) \ge \epsilon p_{\ve{x}_{k-1}}(\tau)$ for all $\tau \in {[m] \choose \beta_k}$. Then we have that 
\begin{equation}
    \tilde{\mathbb{E}}_{\tau_k} \|\ve{x}_k - \ve{x}^*\|^2 \le \left(1 - \frac{\epsilon {m \choose \beta_k}\beta_k \sigma_{\min}^2(A)}{\gamma_k m \sum_{\pi \in {[m] \choose \beta_k}} \|\ve{a}_{t(\pi,\ve{x}_{k-1})}\|^2}\right)\|\ve{x}_{k-1} - \ve{x}^*\|^2 \label{eq:alternateprob}
\end{equation}
where $\tilde{\mathbb{E}}_{\tau_k}$ denotes expectation taken with respect to the sampling of $\tau_k$ according to $\tilde{p}(\tau)$ and conditioned on the choices of $\tau_j$ for $j < k$.  Furthermore, if $\tilde{p}(\tau)$ is constant between iterations (so $\beta_j = \beta$ is constant) and independent of $\ve{x}_{k-1}$, we can iterate the previous result and have 
\begin{equation}
\tilde{\mathbb{E}} \|\ve{x}_k - \ve{x}^*\|^2 \le \prod_{j=1}^k \left(1 - \frac{\epsilon {m \choose \beta}\beta \sigma_{\min}^2(A)}{\gamma_j m \sum_{\pi \in {[m] \choose \beta}} \|\ve{a}_{t(\pi,\ve{x}_{j-1})}\|^2}\right) \|\ve{x}_{0} - \ve{x}^*\|^2 \label{eq:alternateprobiterated}
\end{equation}
where $\tilde{\mathbb{E}}$ denotes expectation taken with respect to all samples of $\tau_j$ for $j = 1, ..., k$. \label{cor:alternateprob}
\end{corollary}

\begin{proof}
    This proof is identical to that of Theorem~\ref{thm:gSKMconvergence} but where we first replace $\mathbb{E}_{\tau_k}$ with $\tilde{\mathbb{E}}_{\tau_k}$ and $p_{\ve{x}_{k-1}}(\tau)$ with $\tilde{p}(\tau)$.  We replace the equation in \eqref{eq:probchangeremark} with an inequality and must add an $\epsilon$ to the numerator of the subtracted term in each line from \eqref{eq:probchangeremark} on.  The iterated bound follows from recursively applying the bounds on the conditional expectations of each iteration.
\end{proof}

\begin{remark} 
\label{rem:unif}
(Uniform probability distribution)
We consider the case of the uniform distribution over samples, i.e., $\tilde{p}(\tau) = 1/{m \choose \beta_k}$.   We note that since $$\frac{\min_{i \in [m]} \|\ve{a}_i\|^2}{{m \choose \beta_k} \max_{i \in [m]} \|\ve{a}_i\|^2} \le p_{\ve{x}_{k-1}}(\tau) \le \frac{\max_{i \in [m]} \|\ve{a}_i\|^2}{{m \choose \beta_k} \min_{i \in [m]} \|\ve{a}_i\|^2}$$ when $min_{i \in [m]} \|\ve{a}_i\|^2 > 0$, we have $$\tilde{p}(\tau) \ge \frac{\min_{i \in [m]} \|\ve{a}_i\|^2}{\max_{i \in [m]} \|\ve{a}_i\|^2} p_{\ve{x}_{k-1}}(\tau).$$  Thus, Corollary~\ref{cor:alternateprob} holds for the uniform distribution with $\epsilon = \min_{i \in [m]} \|\ve{a}_i\|^2/\max_{i \in [m]} \|\ve{a}_i\|^2$.  We note that this additionally provides a convergence analysis for RK of \cite{SV09:Randomized-Kaczmarz} in the case that the matrix $A$ has unnormalized rows and the uniform distribution over rows is employed in sampling.
\end{remark}

The next remarks make simplifying assumptions on the generalized SKM algorithm and our main result to provide better context for comparison with previous works.

\begin{remark} (Recovery of RK guarantees)
If all of the rows of $A$ have equal norm (not necessarily unit norm), then our result specializes to 
\begin{equation}
\mathbb{E}_{\tau_k} \|\ve{x}_k - \ve{x}^*\|^2 \le \Bigg(1 - \frac{\beta_k \sigma_{\min}^2(A)}{\gamma_k \|A\|_F^2}\Bigg)\|\ve{x}_{k-1} - \ve{x}^*\|^2,
\label{eq:equalrows}
\end{equation}
and we can iteratively apply this per-iteration guarantee to give a bound on the error in expectation with respect to all samples, 
\begin{equation}
\mathbb{E} \|\ve{x}_k - \ve{x}^*\|^2 \le \prod_{j=1}^k\Bigg(1 - \frac{\beta_j \sigma_{\min}^2(A)}{\gamma_j \|A\|_F^2}\Bigg)\|\ve{x}_{0} - \ve{x}^*\|^2.
\label{eq:equalrows2}
\end{equation}
Additionally, when $\beta_k = 1$, the sampling distribution~\eqref{eq:probdist} and theoretical error upper bound~\eqref{eq:equalrows} simplifies to the probability distribution and error guarantees of \cite{SV09:Randomized-Kaczmarz}. \label{rem:equalrows}
\end{remark}

\begin{remark} (Improvement of MM guarantees)
Corollary~\ref{cor:SKMconvergence} is obtained from Theorem~\ref{thm:gSKMconvergence} when rows of $A$ have unit norm and $\beta_k = \beta$. When $\beta_k = m$, then an improved convergence rate of
\begin{equation*}
 \|\ve{x}_k - \ve{x}^*\|^2 \le \Bigg(1 - \frac{\sigma_{\min}^2(A)}{\gamma_k}\Bigg)\|\ve{x}_{k-1} - \ve{x}^*\|^2 \le \prod_{j=1}^k \Bigg(1 - \frac{\sigma_{\min}^2(A)}{\gamma_j}\Bigg)\|\ve{x}_{0} - \ve{x}^*\|^2,
\label{eq:MMimprovement}
\end{equation*}
for MM over that shown in~\cite{HN18Motzkin},
\begin{equation*}
 \|\ve{x}_k - \ve{x}^*\|^2 \le \Bigg(1 - \frac{\sigma_{\min}^2(A)}{4\gamma_k}\Bigg)\|\ve{x}_{k-1} - \ve{x}^*\|^2 \le \prod_{j=1}^k \Bigg(1 - \frac{\sigma_{\min}^2(A)}{4\gamma_j}\Bigg)\|\ve{x}_{0} - \ve{x}^*\|^2,
\label{eq:MMoriginal}
\end{equation*}
is obtained. \label{rem:MMoriginal}
\end{remark}

\begin{remark} (Connection to Block RK)
Note that this bound on the convergence rate of SKM additionally provides a bound on the convergence rate of a block Kaczmarz variant. This variant is distinct from the block Kaczmarz method considered in \cite{needell2013paved}.  The analysis of \cite{needell2013paved} requires a pre-partitioned row paving, while the variant considered here allows the blocks to be sampled randomly and not pre-partitioned. Consider the block Kaczmarz variant which in each iteration selects a block of $\beta_k$ rows of $A$, $\tau_k$, and projects the previous iterate into the solution space of the entire block of $\beta_k$ equations.  This variant necessarily converges faster than SKM as it makes more progress in each iteration.  In particular, note that $\{\ve{x} | A_{\tau_k} \ve{x} = \ve{b}_{\tau_k}\} \subset \{\ve{x} | \ve{a}_{t(\tau_k,\ve{y})}^T \ve{x} = b_{t(\tau_k,\ve{y})}\}.$ Given iterate $\ve{x}_{k-1}$ and sample of rows $\tau_k$, let $\ve{x}_k^{\text{SKM}}$ denote the iterate produced by SKM and $\ve{x}_k^{\text{BK}}$ denote the iterate produced by this block Kaczmarz variant.  Note that $\ve{x}_k^{\text{SKM}}$ is the closest point to $\ve{x}_{k-1}$ on the hyperplane associated to equation $t(\tau_k,\ve{x}_{k-1})$ so, since $\ve{x}_k^{\text{BK}}$ also lies on this hyperplane, we have $$\|\ve{x}_k^{\text{SKM}} - \ve{x}_{k-1}\|^2 \le \|\ve{x}_k^{\text{BK}} - \ve{x}_{k-1}\|^2.$$  Now, we note that by orthogonality of the projections, we have $$\|\ve{x}_k^{\text{SKM}} - \ve{x}_{k-1}\|^2 + \|\ve{x}_k^{\text{SKM}} - \ve{x}^*\|^2 = \|\ve{x}_{k-1} - \ve{x}^*\|^2 = \|\ve{x}_k^{\text{BK}} - \ve{x}_{k-1}\|^2 + \|\ve{x}_k^{\text{BK}} - \ve{x}^*\|^2$$ so by the above inequality, we have $$\|\ve{x}_k^{\text{BK}} - \ve{x}^*\|^2 \le \|\ve{x}_k^{\text{SKM}} - \ve{x}^*\|^2.$$
A visualization of this situation is presented in Figure~\ref{fig:blockkacz}.  Thus, the progress made by BK in any fixed iteration is at least as large as the progress made by SKM, so it must converge at least as quickly.
\end{remark}

\begin{figure}
    \centering
    \includegraphics[width=0.6\textwidth]{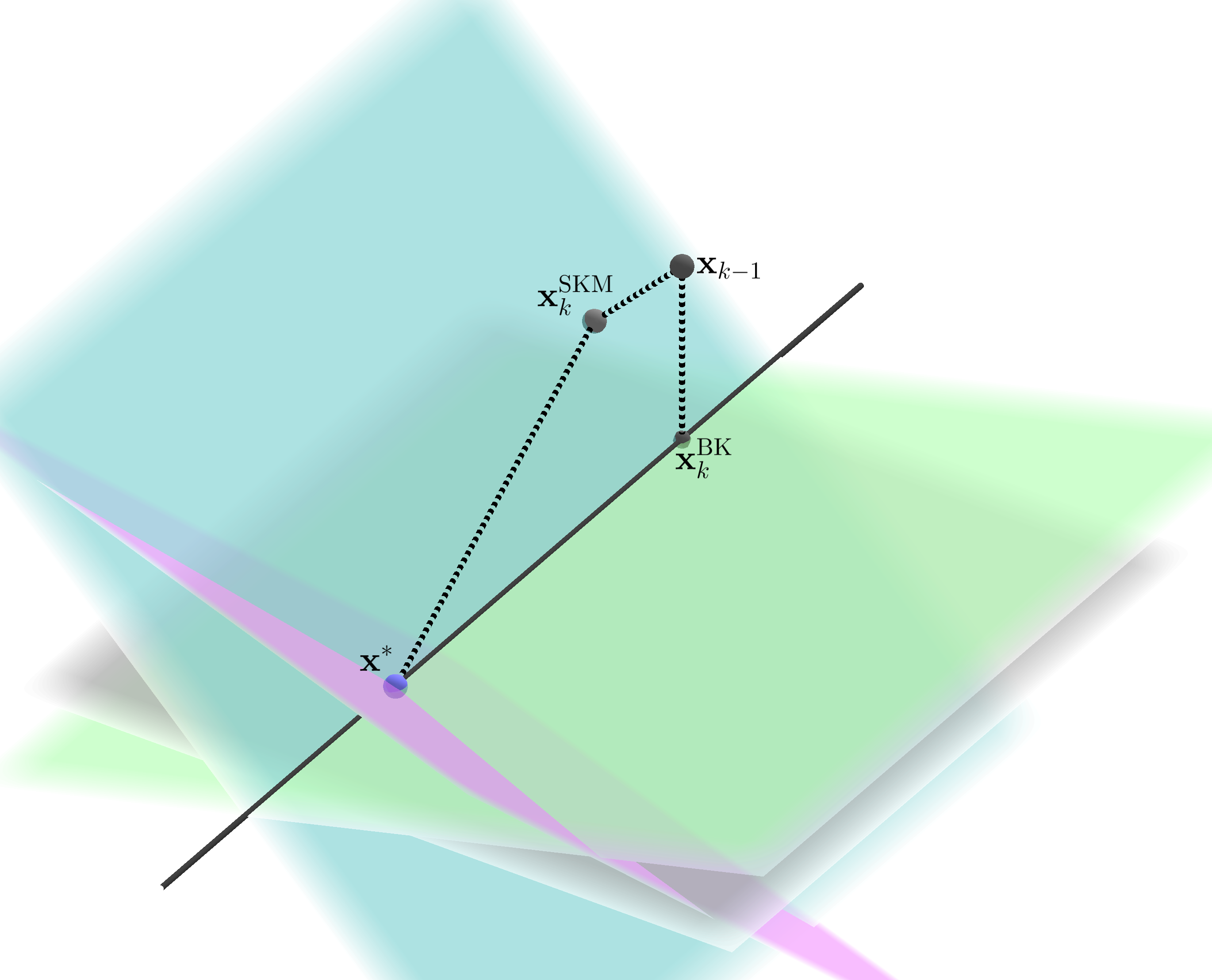}
    \caption{The SKM and BK iterates, $\ve{x}_k^{\text{SKM}}$ and $\ve{x}_k^{\text{BK}}$, generated by one iteration starting at $\ve{x}_{k-1}$ satisfy $\|\ve{x}_k^{\text{BK}} - \ve{x}^*\|^2 \le \|\ve{x}_k^{\text{SKM}} - \ve{x}^*\|^2$.}
    \label{fig:blockkacz}
\end{figure}

One may be assured that the contraction term in Theorem~\eqref{thm:gSKMconvergence} is always strictly positive.  We prove this simple fact in Proposition~\ref{prop:gammabound}.

\begin{proposition}
For any matrix $A$ defining a consistent system with $\ve{x}^* = A^\dagger \ve{b}$ and $\ve{x}_{j-1} \in \text{range}(A^\top)$, we have $$\gamma_j \ge \frac{\beta_j {m \choose \beta_j} \sigma_{\min}^2(A)}{m \sum_{\tau \in {[m] \choose \beta_j}}\|\ve{a}_{t(\tau,\ve{x}_{j-1})}\|^2}.$$ \label{prop:gammabound}
\end{proposition}

\begin{proof}
    Beginning with the definition of $\gamma_j$, we have
\begin{align*}
    \frac{\sum_{\tau_j \in {[m] \choose \beta_j}} \|A_{\tau_j}\ve{x}_{j-1} - \ve{b}_{\tau_j}\|^2}{\sum_{\tau_j \in {[m] \choose \beta_j}} \|A_{\tau_j}\ve{x}_{j-1} - \ve{b}_{\tau_j}\|_\infty^2} &= \frac{\frac{\beta_j}{m} {m \choose \beta_j} \|A(\ve{x}_{j-1} - \ve{x}^*)\|^2}{\sum_{\tau_j \in {[m] \choose \beta_j}} |\ve{a}_{t(\tau_j,\ve{x}_{j-1})}^\top (\ve{x}_{j-1} - \ve{x}^*)|^2}
    \\&\ge \frac{\frac{\beta_j}{m} {m \choose \beta_j} \sigma_{\min}^2(A) \|\ve{x}_{j-1} - \ve{x}^*\|^2}{\sum_{\tau_j \in {[m] \choose \beta_j}} \|\ve{a}_{t(\tau_j,\ve{x}_{j-1})}\|^2 \|\ve{x}_{j-1} - \ve{x}^*\|^2}
    \\&= \frac{\beta_j {m \choose \beta_j} \sigma_{\min}^2(A)}{m \sum_{\tau \in {[m] \choose \beta_j}}\|\ve{a}_{t(\tau,\ve{x}_{j-1})}\|^2},
\end{align*}
where the inequality follows from properties of singular values and Cauchy-Schwartz.
\end{proof}

Because Theorem~\ref{thm:gSKMconvergence} shows that the contraction coefficient for generalized SKM is dependent on the dynamic range, the following section discusses bounds on the dynamic range for special types of linear systems.

\section{Analysis of the Dynamic Range}
Since the dynamic range plays an integral part in the convergence behavior for generalized SKM, the dynamic range is analyzed here for different types of specialized linear systems. Note that the dynamic range has also appeared in other works, although not under the guise of ``dynamic range". For example, in~\cite{bai2018greedy} the authors proposed a Greedy Randomized Kaczmarz (GRK) algorithm that finds a subset of indices to randomly select the next row to project onto. The operation of finding this subset relies on a ratio between the $\ell_\infty$ and $\ell_2$ norms of the residual at the current iteration, essentially using a proxy of the dynamic range. In the next section, we analyze the dynamic range for random Gaussian linear systems and remark on the extension to other random linear systems. {In the following section, we analyze the dynamic range for linear systems encoding average consensus problems on directed graphs via the incidence matrix.}

\subsection{Gaussian Matrices} When entries of the measurement matrix $A$ are drawn i.i.d. from a standard Gaussian distribution, it can be shown that the dynamic range is upper bounded by $\mathcal{O} (n\beta/\log \beta)$. The proof of the upper bound of $\gamma_k$ is similar to Lemma 2 of ~\cite{HN18Motzkin}, where the authors analyze the dynamic range for $\beta=m$. Here, we generalized the bound for varying samples sizes $\beta_k$.

\begin{proposition} Let $A \in \mathbb{R}^{m\times n}$ be a random Gaussian matrix with $a_{ij} \sim \mathcal{N}(0,\sigma^2)$. For each subset $\tau \in {[m] \choose \beta_k}$, let $I_\tau \subseteq \tau$ denote the set of rows in $\tau$ that are independent of $\ve{x}$ and note $|I_\tau| \leq \beta_k$. Assuming there is at least $m^\prime$ rows in $[m]$ which are independent of $\ve{x}$, the dynamic range can be upper bounded as:
\begin{equation}
    \gamma_j = \frac{\sum_{\tau \in {[m] \choose \beta}} \mathbb{E}_a\|A_{\tau}\ve{x}\|^2}{\sum_{\tau \in {[m] \choose \beta}} \mathbb{E}_a\|A_{\tau}\ve{x}\|_\infty^2} \leq \frac{{m \choose \beta_k} \left( \beta_k n + \sum_{i \in \tau \setminus I_{\tau} }\|\ve{a}_{i}\|^2/\sigma^2\right) }{{m^\prime \choose \beta_k} \log(\beta_k)}.
    \label{eq:gaussGamma}
\end{equation}
\end{proposition}
\begin{remark}
Note that the factor ${m \choose \beta_k} / {m^\prime \choose \beta_k}$ is $\mathcal{O}(1)$ as $m \rightarrow \infty$ since
\begin{equation*}
    \frac{{m \choose \beta_k}}{{m^\prime \choose \beta_k}} = \frac{m!}{\beta_k!(m-\beta_k)!} \frac{\beta_k! (m - j - \beta_k)!}{(m-j)!} = \prod_{i = 0}^j  \frac{m-i}{m-\beta_k - i}.
\end{equation*} Thus, we conclude that the expected dynamic range for any iteration $k$ is $\mathcal{O} (n\beta_k / \log(\beta_k)).$
\end{remark}
\begin{proof} Without loss of generality, we let the solution to the system $\ve{x}^* = 0$ so that $\ve{b} = 0$.  We are then interested in finding an upper bound on the dynamic range~\eqref{eq:gammaj} in expectation. Here, the expectation is taken with respect to the random i.i.d. draws of the entries of $A$. To that end, we derive upper bounds and lower bounds on the numerator and denominator of~\eqref{eq:gammaj}. Starting with the upper bound on the numerator we have 
\begin{align*}
    \sum_{\tau \in {[m] \choose \beta_k}} \mathbb{E}_a\|A_{\tau}\ve{x} \|^2 &\leq \sum_{\tau \in {[m] \choose \beta}} \sum_{i \in \tau }\mathbb{E}_a\left(\|\ve{a}_{i}\|^2 \| \ve{x} \|^2\right) \\
    &= \sum_{\tau \in {[m] \choose \beta_k}} \left( \sum_{i \in I_\tau} \mathbb{E}_a\|\ve{a}_{i}\|^2 \| \ve{x} \|^2 + \sum_{i \in \tau \setminus I_{\tau} }\|\ve{a}_{i}\|^2 \| \ve{x} \|^2 \right)\\
    &= \sum_{\tau \in {[m] \choose \beta_k}} \left( \sum_{i \in I_\tau} n\sigma^2 \| \ve{x} \|^2 + \sum_{i \in \tau \setminus I_{\tau} }\|\ve{a}_{i}\|^2 \| \ve{x} \|^2 \right) \\
    &\leq \sum_{\tau \in {[m] \choose \beta_k}} \left( \beta_k n\sigma^2 + \sum_{i \in \tau \setminus I_{\tau} }\|\ve{a}_{i}\|^2  \right)\| \ve{x} \|^2 \\
    &= {m \choose \beta_k} \left( \beta_k n\sigma^2 + \sum_{i \in \tau \setminus I_{\tau} }\|\ve{a}_{i}\|^2  \right)\| \ve{x} \|^2 
\end{align*}
where the first inequality follows from the Cauchy-Schwartz inequality and remaining computation uses the fact that $\mathbb{E}_a \|\ve{a}_i \|^2 = n\sigma^2$ and simplifies the expression. The lower bound follows from
\begin{align*}
   \sum_{\tau \in {[m] \choose \beta}} \mathbb{E}_a\|A_{\tau}\ve{x}\|_\infty^2 
   &= 
   \sum_{\tau \in {[m] \choose \beta}} \mathbb{E}_a \max_{i\in\tau} \langle\ve{a}_i, \ve{x} \rangle^2 \geq \sum_{\tau \in {[m] \choose \beta}} \mathbb{E}_a \max_{i\in I_\tau} \langle\ve{a}_i, \ve{x} \rangle^2 \\ 
   &\geq \sum_{\tau \in {[m] \choose \beta}} \left(\mathbb{E}_a \max_{i\in I_\tau} \langle\ve{a}_i, \ve{x} \rangle\right)^2 
    \geq \sum_{\substack{\tau \in {[m] \choose \beta} \\ |I_\tau| = \beta_k}} \left(\mathbb{E}_a \max_{i\in I_\tau} \langle\ve{a}_i, \ve{x} \rangle\right)^2 \\
   &\geq \sum_{\substack{\tau \in {[m] \choose \beta} \\ |I_\tau| = \beta_k}} \sigma^2 \|\ve{x} \|^2 \log(\beta_k)\\
   &\geq {m^\prime \choose \beta_k} \sigma^2 \|\ve{x} \|^2 \log(\beta_k),
\end{align*}
where the second to last inequality uses the fact that for i.i.d. Gaussian random variables $g_1, g_2, ... g_N \sim \mathcal{N}(0, \sigma^2)$, we have that $\mathbb{E}(\max_{ i \in [N]} g_i) \gtrsim \sigma \sqrt{\log N}$ and that $\langle \ve{a}_i, \ve{x} \rangle \sim \mathcal{N}(0, \sigma^2 \|\ve{x}\|^2)$. 

Therefore, we have
\begin{align*}
    \gamma_k &= \frac{\sum_{\tau \in {[m] \choose \beta}} \mathbb{E}_a\|A_{\tau}\ve{x}\|^2}{\sum_{\tau \in {[m] \choose \beta}} \mathbb{E}_a\|A_{\tau}\ve{x}\|_\infty^2} \\
    & \leq \frac{{m \choose \beta_k} \left( \beta_k n\sigma^2 + \sum_{i \in \tau \setminus I_{\tau} }\|\ve{a}_{i}\|^2\right) }{{m^\prime \choose \beta_k} \sigma^2 \log(\beta_k)}.
\end{align*}
Dividing all terms by $\sigma^2$ attains the desired result \eqref{eq:gaussGamma}.
\end{proof}

We conjecture that the true bound is actually $\mathcal{O}(\beta_k/\log(\beta_k))$ and that the $n$ is an artifact of our proof technique; {throughout our experiments varying $n$ (and for various $m$), we have not found any dependence of $\gamma_k$ on $n$.  For this reason, w}e have plotted $\gamma_k$ and the corresponding conjectured bound in the left of Figure~\ref{fig:dynranges} for a Gaussian matrix of size $50000 \times 500$.

\begin{remark} To extend to other distributions, one can simply note that as the signal dimension $n$ gets large, the Law of Large numbers can be invoked and a similar computation can be used to show an upper bound on the dynamic range of the system.
\end{remark}

\begin{figure}
    \centering
    \includegraphics[width=0.49\textwidth]{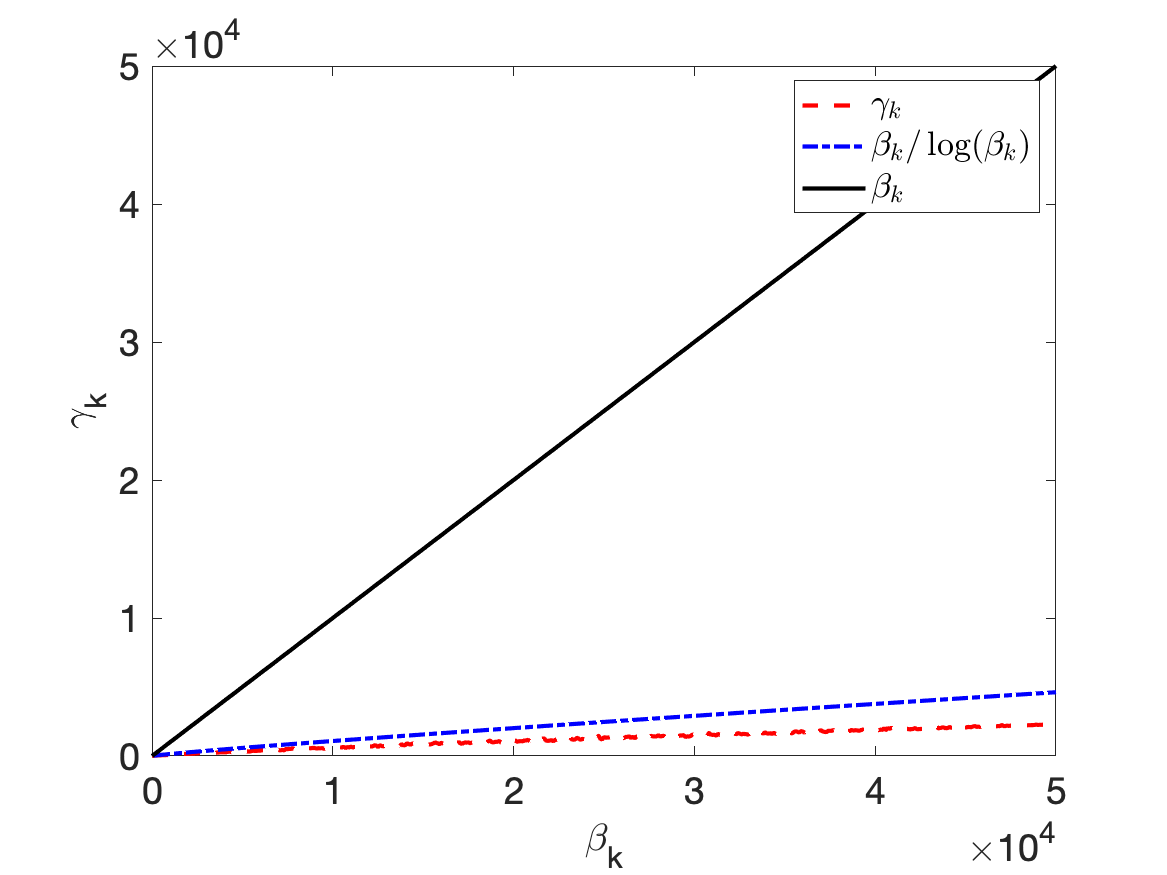}\includegraphics[width=0.49\textwidth]{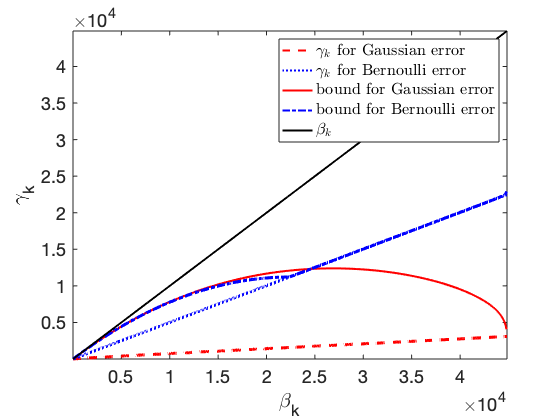}
    \caption{Dynamic ranges $\gamma_k$ for various sample sizes $\beta_k$.  Left: Gaussian matrix $A \in \mathbb{R}^{50000 \times 500}$ and Gaussian error ${\ve{e}_k} \in \mathbb{R}^{500}$ (red). Conjectured bound is plotted in blue. Right: Incidence matrix $Q \in \mathbb{R}^{44850 \times 300}$ for complete graph $K_{300}$ with Gaussian error ${\ve{e}_k} \in \mathbb{R}^{300}$ (red) and a sparse error ${\ve{e}_k}$ with Bernoulli random variable entries (blue). Bounds are plotted in the same colors with different line styles.}
    \label{fig:dynranges}
\end{figure}

\subsection{Incidence Matrices} \label{sec:avgcons}
In the previous subsection, we analyzed the dynamic range for systems with measurement matrices that are randomly generated. Deterministically generated measurement matrices are additionally of interest.  In this subsection, we analyze the dynamic range associated to incidence matrices of {undirected} graphs, $\mathcal{G} = (\mathcal{V},{\mathcal{E}})$.  The incidence matrix $Q$ associated to {an undirected} graph is of size ${|\mathcal{E}|} \times |\mathcal{V}|$.  For each {edge}, $(i,j) \in \mathcal{E}$ which connects vertex $i$ to vertex $j$, the associated row of $Q$ is all zeros with a one and negative one in the $i$th and $j$th entries.  These types of matrices arise in {one formulation of} the \emph{average consensus} problem {as a system of linear equations}.  

The average consensus problem on a graph asks that all nodes on the graph learn the average value of initial, secret values held by each node using only local information; that is, each node $i$ initially knows $c_i$ and at solution they should all know $\frac{1}{|\mathcal{V}|}\sum_{i \in \mathcal{V}} c_i$ with communication only across edges.  This problem models computation in many real life applications such as clock synchronization, localization without GPS, distributed data fusion in sensor networks, and load balancing. Many analyses of (asynchronous and synchronous) distributed methods for this problem exploit its formulation as a system of linear equations.  The problem over a directed graph may be formulated as a linear system using either the incidence matrix (described above) or the Laplacian matrix, $L = D - A$ where $D$ is the diagonal matrix of node degrees and $A$ is the adjacency matrix, or more generally as an \emph{average consensus system} defined in \cite{loizou2016new}.

The \emph{gossip methods} that solve the average consensus problem are generalized by the Kaczmarz methods \cite{loizou2019revisiting}. Early work making this connection focused on the formulation of the average consensus problem as a Laplacian system \cite{zouzias2015randomized}, but subsequent work generalized this connection to systems formulated more generally \cite{loizou2016new}.
RK specializes to the \emph{randomized gossip} method in which the pair of nodes which update are selected at random \cite{xiao2004fast,boyd2006randomized}. The connection between gossip methods and other Kaczmarz variants have been observed; the connection to block methods was noted in \cite{loizou2016new}, extended methods in \cite{zouzias2015randomized}, and accelerated methods in \cite{loizou2018accelerated,loizou2019provably}.  This connection has also spawned new gossip methods; in \cite{hanzely2017privacy} the authors propose a privacy preserving gossip method, and in \cite{aybat2017decentralized}, the authors propose an accelerated, decentralized gossip method.  In \cite{loizou2019revisiting}, the authors summarize many of these advances and connections between the Kaczmarz literature and gossip literature.  They first noted and exploited the fact that SKM specializes to a variant of \emph{greedy gossip with eavesdropping (GGE)} in which the nodes are selected from amongst a random sample to maximize the update \cite{ustebay2010greedy}.  Thus, our analysis provides an alternate convergence rate for GGE.

Now, we consider the dynamic range for an incidence matrix.  We can derive a simple bound on the dynamic range in each iteration that depends only upon the entries of the current error vector, $\ve{e}_k := \ve{x}_k - \ve{x}^*$.  In particular,
\begin{multline}
    \gamma_k = \frac{\sum_{\tau \in {[m] \choose \beta_k}} \|Q_\tau \ve{e}_k\|^2}{\sum_{\tau \in {[m] \choose \beta_k}} \|Q_\tau \ve{e}_k\|_\infty^2} = \frac{{m \choose \beta_k} \frac{\beta_k}{m} \sum_{(i,j) \in \mathcal{E}} ({e}_k^{(i)} - {e}_k^{(j)})^2}{\sum_{\tau \in {[m] \choose \beta_k}} \max_{(i,j) \in \tau} ({e}_k^{(i)} - {e}_k^{(j)})^2} 
    \\\le \frac{\beta_k (m-\beta_k+1) \sum_{(i,j) \in \mathcal{E}} ({e}_k^{(i)} - {e}_k^{(j)})^2}{m \sum_{n=\beta_k}^{|\mathcal{E}|}({e}_k^{(n_i)} - {e}_k^{(n_j)})^2},
\end{multline}
where $n_i$ and $n_j$ denote the vertices connected by the $n$th smallest magnitude difference across an edge.  This bound improves for iterates with a sufficient amount of variation in the coordinates.  We have plotted $\gamma_k$ and the corresponding bounds in the right of Figure~\ref{fig:dynranges}.  We calculate these values for the incidence matrix $Q \in \mathbb{R}^{44850 \times 300}$ of the complete graph $K_{300}$ in the cases when the error is a Gaussian vector (red) and a Bernoulli vector (blue).

\begin{proposition}
If the right-hand-side vector associated to the system $Q\ve{x} = \ve{b}$ is $\ve{b} = \ve{0}$, as in the average consensus problem, then this bound on the dynamic range is easily computed from the current iterate, 
$$\gamma_k \le \frac{\beta_k (m-\beta_k+1) \sum_{(i,j) \in \mathcal{E}} (x_k^{(i)} - x_k^{(j)})^2}{m \sum_{n=\beta_k}^{|\mathcal{E}|}(x_k^{(n_i)} - x_k^{(n_j)})^2}.$$
\end{proposition}

\begin{remark} We note that this bound on the dynamic range holds for any incidence matrix $Q$, including those associated with directed graphs.  In the case of directed graphs, however, additional assumptions must be made to ensure the well-posedness of the average consensus problem.  Additionally, the Kaczmarz methods must be altered to ensure communication in only one direction along edges for directed graphs.  Analyses of regular Kaczmarz methods, such as RK or SKM, do not apply to average consensus systems on directed graphs.  We leave consideration of Kaczmarz type methods for this variant of the average consensus problem to future work.\end{remark}

\section{Experiments}

In this section, we present simulated and real world experiments using SKM for varying sample sizes $\beta$. In the simulated experiments, we compare the theoretical convergence guarantees to the empirical performance of SKM, measured by approximation error $\|\ve{e}_j\|^2$, averaged over 20 random trials. The number of rows $m = 50000$ and number of columns $n = 500$ are fixed for all simulated experiments. The solution to the system is a vector $\ve{x}^* \in \mathbb{R}^n$ where each entry is drawn i.i.d. from a standard Gaussian distribution. In each experiment, the systems are consistent so that $\ve{b} = A\ve{x}^*$. The sample sizes considered for this experiment are $\beta = \{1, 100, 200, 500, 1000 \}$. Unless otherwise stated, the rows of $A$ are uniformly selected without replacement. The experiments presented in this section were are performed in
MATLAB 2017b on a MacBook Pro 2015 with a 2.7 GHz Dual-Core
Intel Core i5 and 8GB RAM. For practical reasons, we normalize the rows of $A$ and utilize the bound shown in Corollary~\ref{cor:SKMconvergence}.  

\begin{figure}
    \centering
    \includegraphics[width=.32\textwidth]{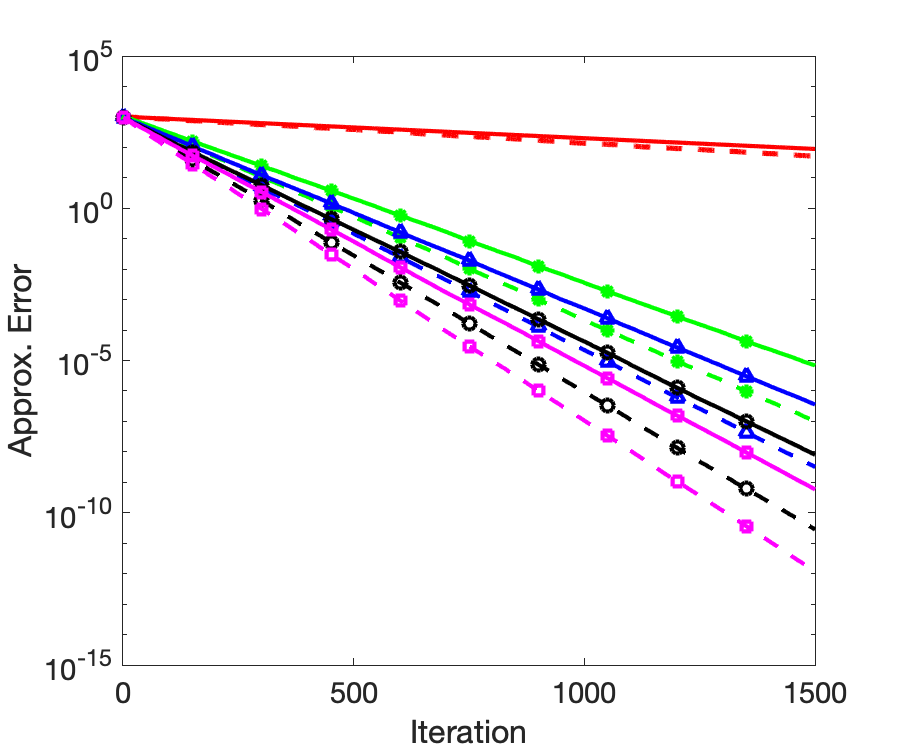}
     \includegraphics[width=.32\textwidth]{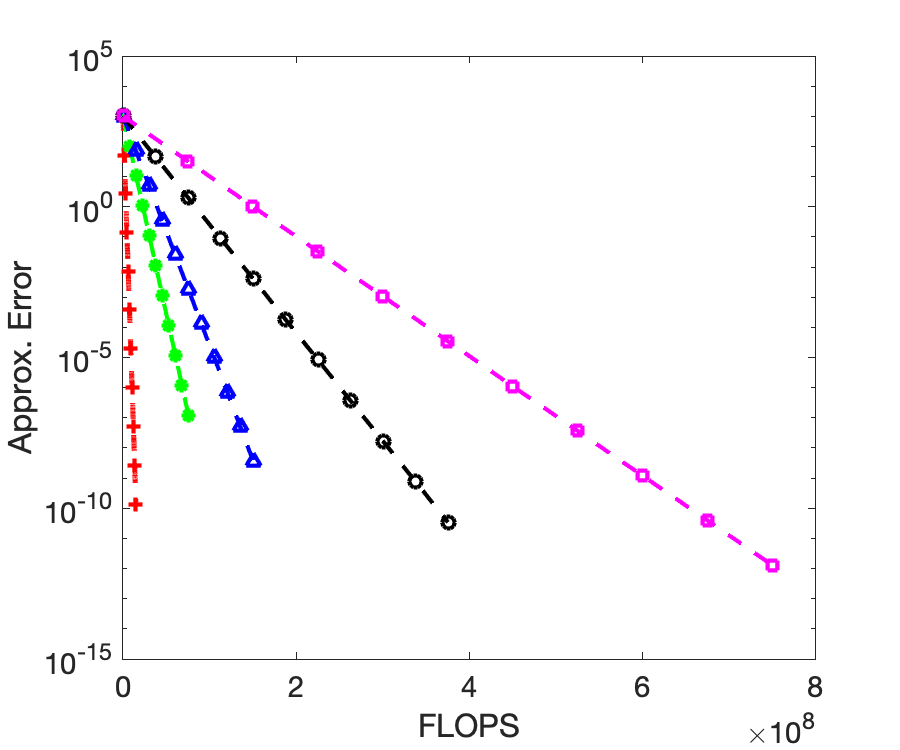}
      \includegraphics[width=.32\textwidth]{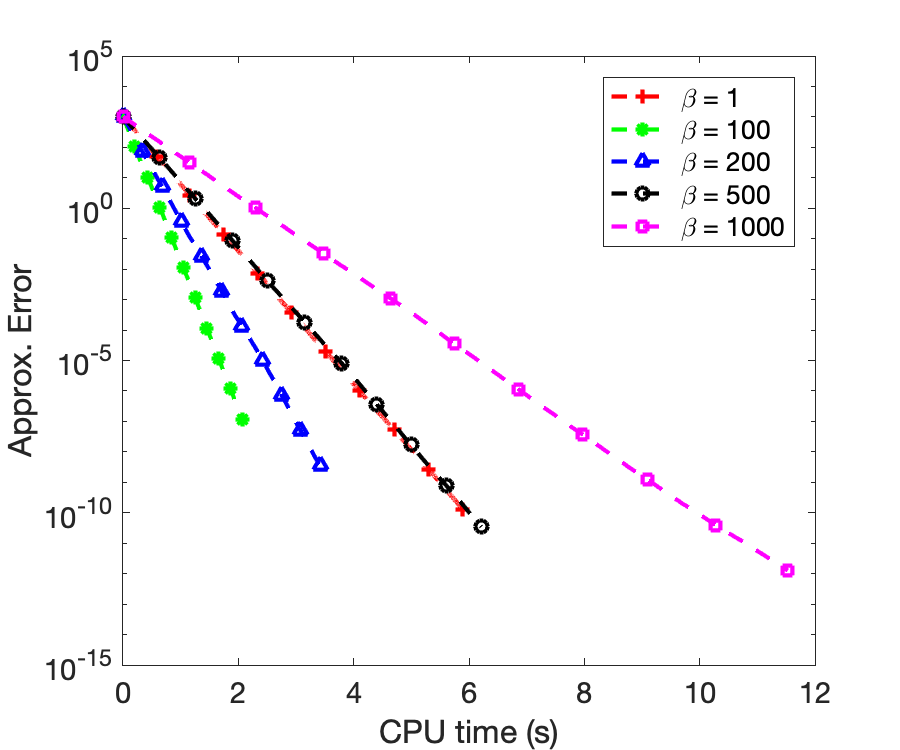}
    \caption{Comparison of SKM for various choices of fixed $\beta$ values on linear system with entries of $A$ drawn from i.i.d. from $\mathcal{N}(0,1)$. (left) Iteration vs Approximation Error with dashed lines representing average empirical performance of SKM and solid lines representing theoretical upper bounds for SKM. (middle) FLOPS vs Approximation Error. (right) CPU time vs Approximation Error}
    \label{fig:gauss}
\end{figure}

Figure~\ref{fig:gauss} and Figure~\ref{fig:unif} show the results for Gaussian and Uniform random matrices $A$ respectively. For Gaussian random matrices, each entry of $A$ is drawn i.i.d. from a standard Gaussian distribution. For Uniform random matrices, entries of $A$ are drawn i.i.d. uniformly from the interval $[0,1]$. In each figure, we plot along the horizontal axis the (left subplot) iteration, (middle subplot) FLOPS or floating point operations, and (right subplot) CPU time in seconds. The vertical axis for all plots indicate the average approximation error across random trials. Note that the left most subplot for both figures also contains a solid line, which indicates the theoretical upper bound of the algorithm provided by Corollary~\ref{cor:SKMconvergence}. 

For linear systems with Gaussian random matrices, we see in Figure~\ref{fig:gauss} that the convergence upper bound proven in this work closely matches the behavior of the SKM algorithm regardless of the choice of sample size $\beta$. To compare this result to previous works, note that when $\beta=1$, the upper bound provided in Corollary~\ref{cor:SKMconvergence} simply recovers the previous known upper bound for SKM with normalized rows, a bound which was completely independent of $\beta$. In other words, the solid red line is the comparative previous known SKM upper bound for all $\beta$.

Of course, choices of large sample sizes $\beta$ come at a cost, which are captured in the middle and right most subplots of Figure~\ref{fig:gauss}. When measuring efficiency, it seems that $\beta=1$ makes the most progress with minimal FLOPS while $\beta = 100$ is optimal amongst the tested sample sizes with respect to CPU time. This difference is typically explained by the programming and computer architecture (e.g., it may be more efficient to work on batches of rows as opposed to single rows at a time).

Figure~\ref{fig:unif} uses a uniform random matrix $A$ instead of a Gaussian random matrix. While the algorithm efficiency with respect to FLOPS and CPU time have similar conclusions to those in the Gaussian measurement matrix case (as one would expect), the iteration vs approximation error plot now tells a different story. Unlike in the Gaussian case, the theoretical upper bound no longer closely tracks the approximation error of SKM. The looseness here comes from lower bounding the norm of $\| A\ve{x} \|_2^2$ with the magnitude of $\ve{x}$ times the smallest nonzero singular value of $A$ squared. Empirically, we have seen that this lower bound is tighter for Gaussian systems than Uniform systems. It should be noted that even though our theoretical bounds do not track the approximation error for SKM as tightly, they are still a slight improvement over the previous known bounds for SKM.

\begin{figure}
    \centering
    \includegraphics[width=.32\textwidth]{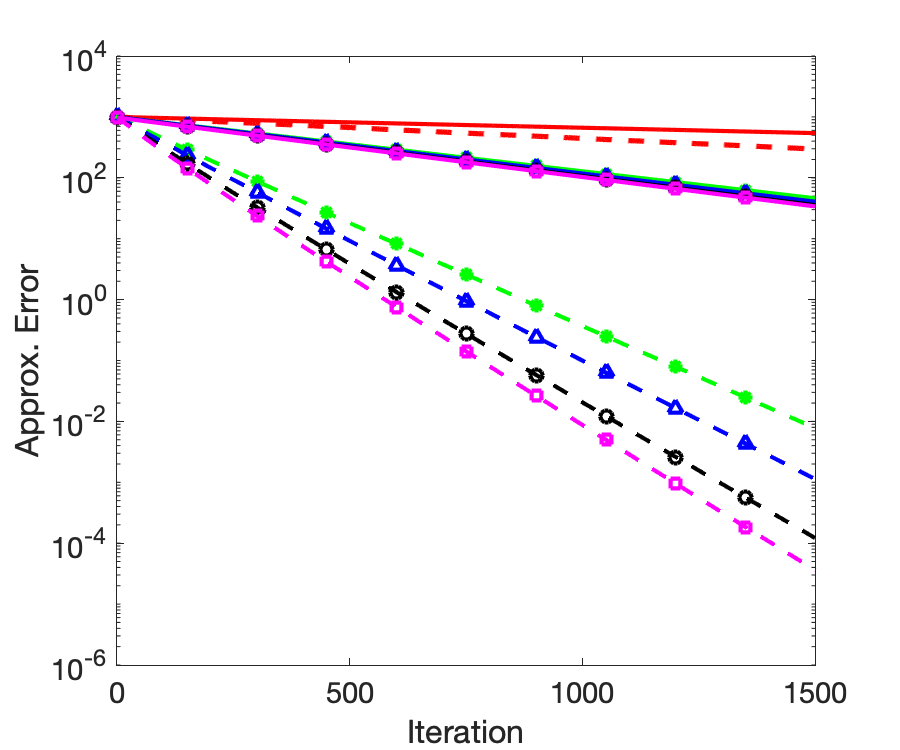}
     \includegraphics[width=.32\textwidth]{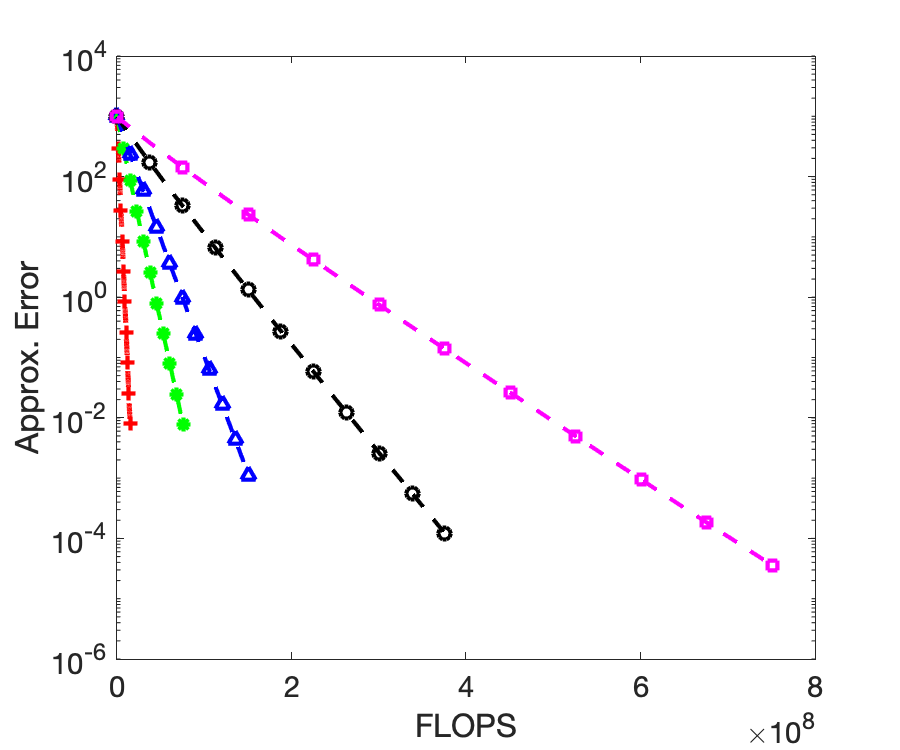}
      \includegraphics[width=.32\textwidth]{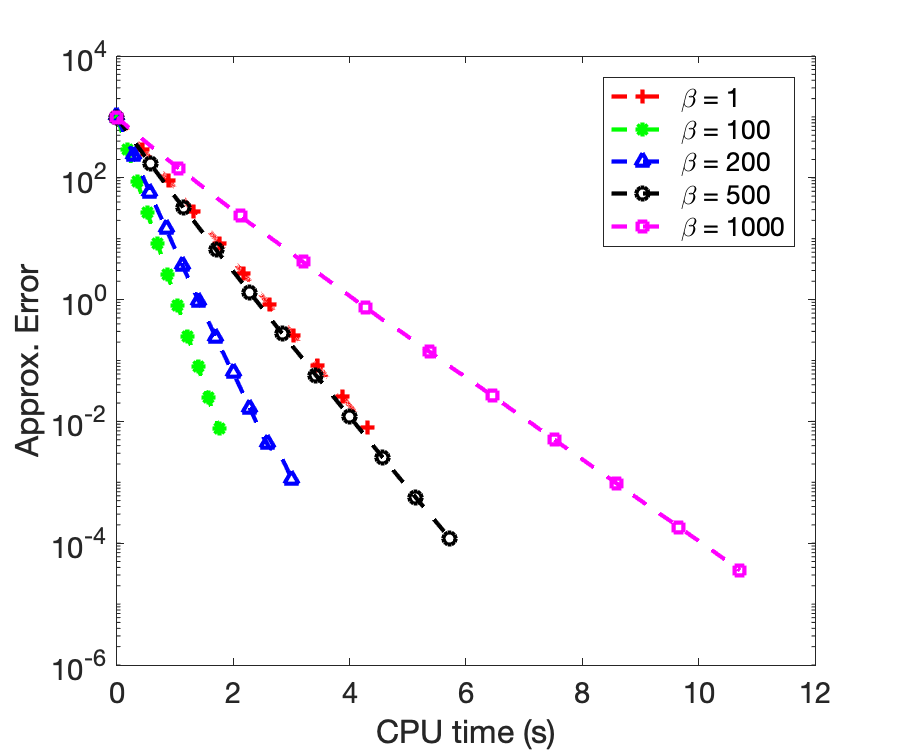}
\caption{Comparison of SKM for various choices of fixed $\beta$ values on linear system with entries of a $A$ drawn from $\text{unif}([0,1])$. (left) Iteration vs Approximation Error with dashed lines representing average empirical performance of SKM and solid lines representing theoretical upper bounds for SKM. (middle) FLOPS vs Approximation Error. (right) CPU time vs Approximation Error}
    \label{fig:unif}
\end{figure}

In addition to being an improvement over the previously known SKM bound, the convergence bound shown in this work enjoys the flexibility of being amenable to a dynamically selected sample size $\beta_k$. Figure~\ref{fig:dynamic} and Figure~\ref{fig:dynamic_unif} show the empirical results from experiments where $\beta_k$ is allowed to change at every iteration. In Figure~\ref{fig:dynamic} the measurement matrix $A$ is again a random Gaussian matrix and in Figure~\ref{fig:dynamic_unif} the measurement matrix entries are drawn i.i.d.\ from Unif$([0,1])$. We consider three sampling regimes that change $\beta_k$ at every iteration: `useDynRng' which allocates $\beta_k$ as a function of the dynamic range, `slowInc' which increases $\beta_k$ at every iteration until $\beta_k = m$, and finally `rand' which uniformly at random selects a $\beta_k \in [m]$ at every iteration. More specifically, the `useDynRng' uses the heuristic  $\beta_k = \lceil \max(m, \frac{m \|A_{\tau_{k-1}}\ve{x}_{k-1} - \ve{b}_{\tau_{k-1}} \|_\infty}{n \|A_{\tau_{k-1}}\ve{x}_{k} - \ve{b}_{\tau_{k-1}} \|_2} ) \rceil$. Note that this choice of $\beta_k$ relies directly on the inverse of an \emph{approximation} of the dynamic range $\gamma_k$ computed without incurring additional computational cost, in a naive attempt to optimize the contraction term of the theoretical bound for SKM. Even though $\beta_k$ changes at each iteration, we see that the theoretical guarantees proven in this work still track the progress of SKM. This indeed opens up new and interesting avenues of research including how one can compute an optimal $\beta_k$ at every iteration. Since the focus of this work is the improvement of the convergence bound of SKM, we leave this for future work.

\begin{figure}
    \centering
    \includegraphics[width=.32\textwidth]{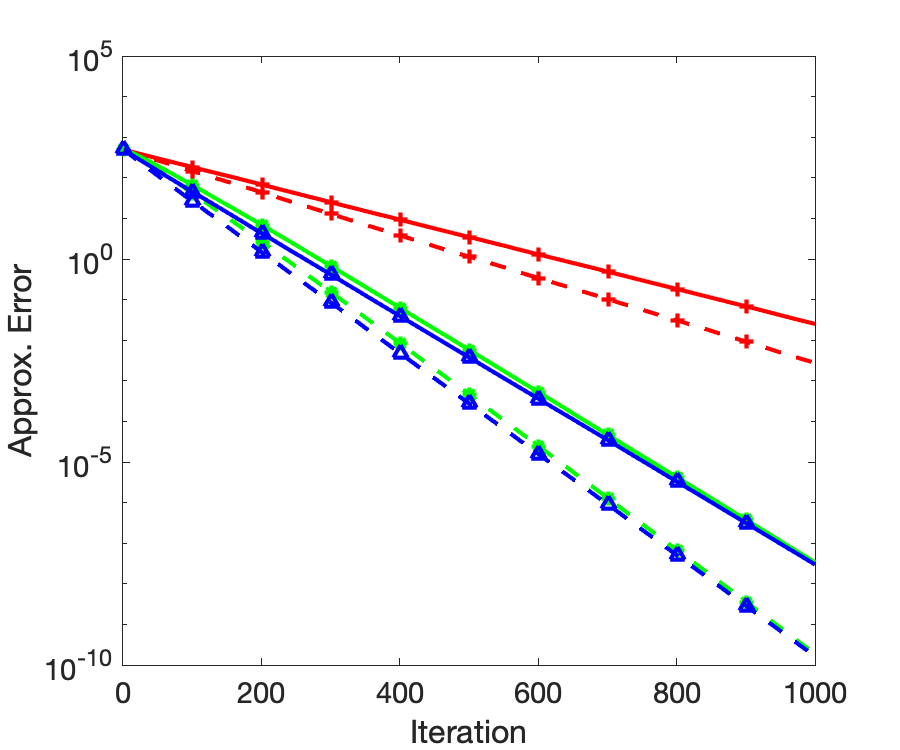}
    \includegraphics[width=.32\textwidth]{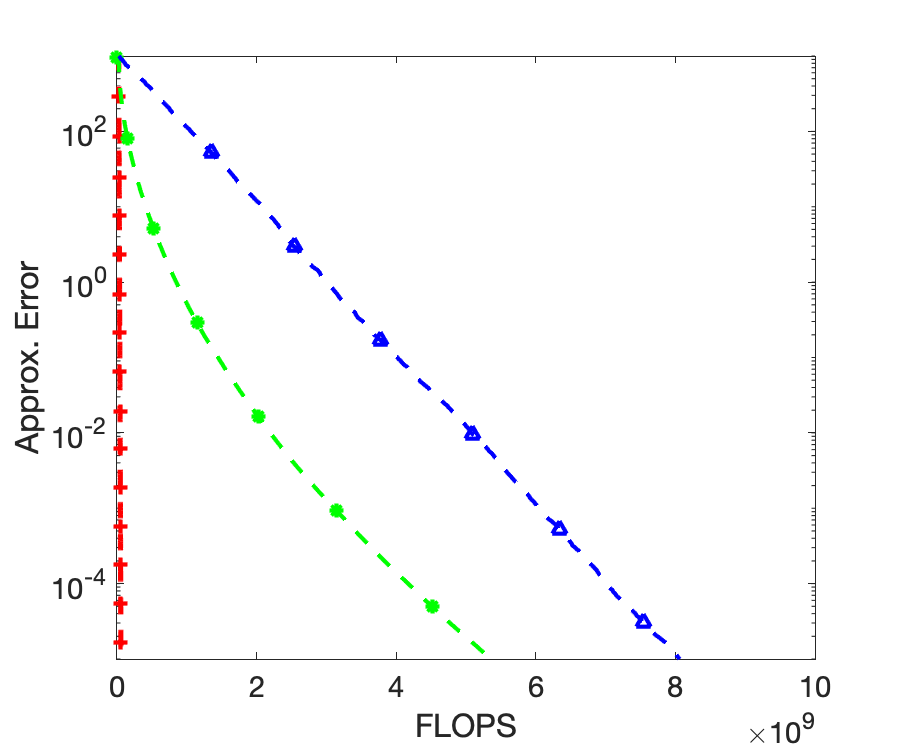}
    \includegraphics[width=.32\textwidth]{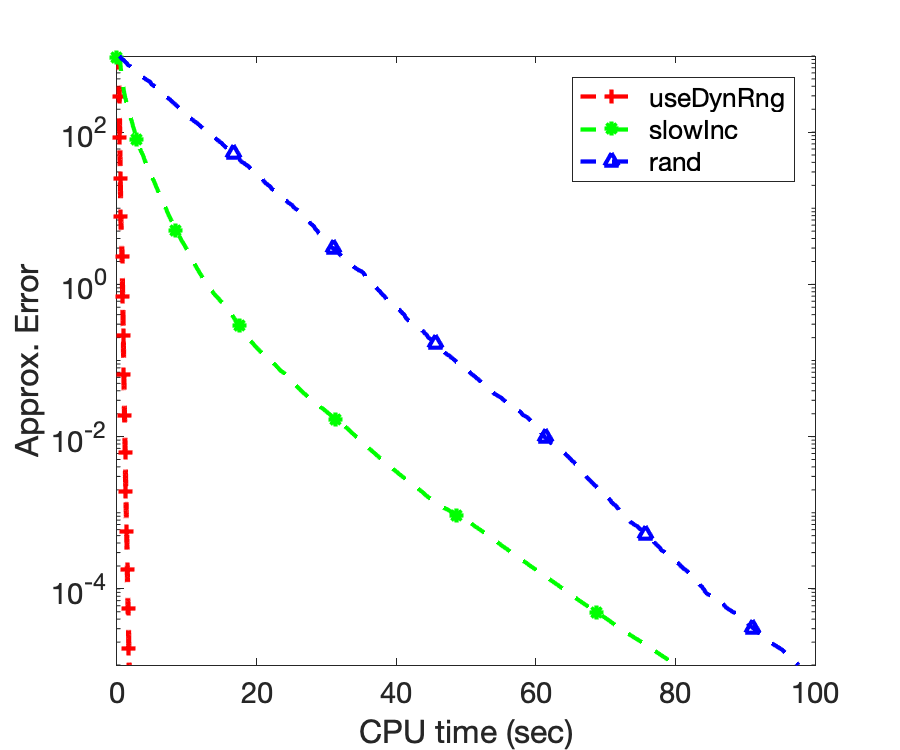}
\caption{Comparison of SKM for various choices of dynamically selected $\beta_k$ values on linear system with entries of $A$ drawn from $\mathcal{N}(0,1)$. }
    \label{fig:dynamic}
\end{figure}

\begin{figure}
    \centering
    \includegraphics[width=.32\textwidth]{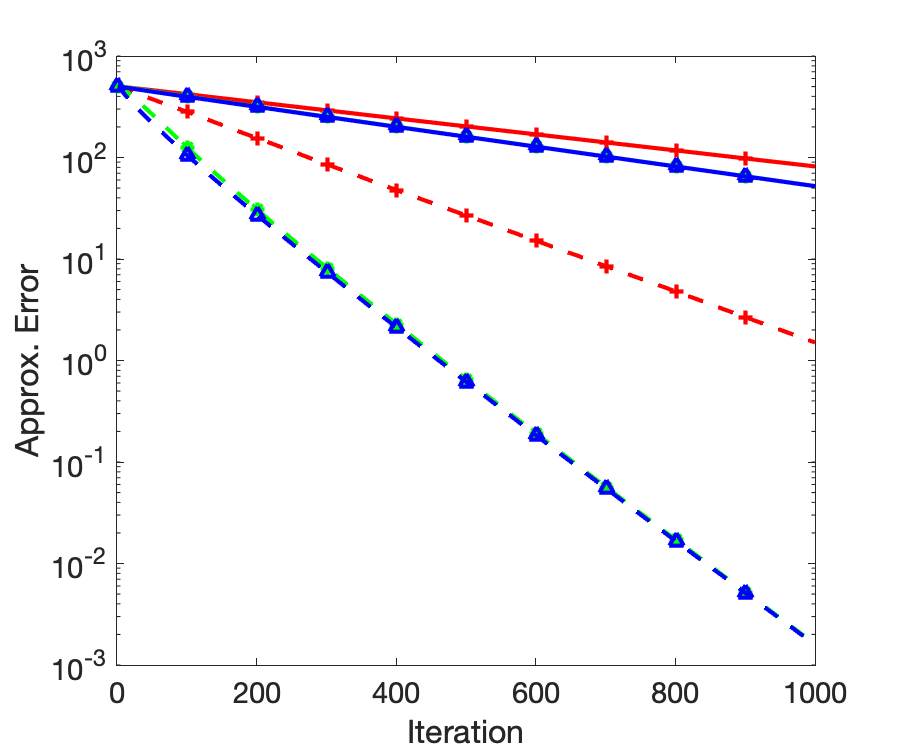}
    \includegraphics[width=.32\textwidth]{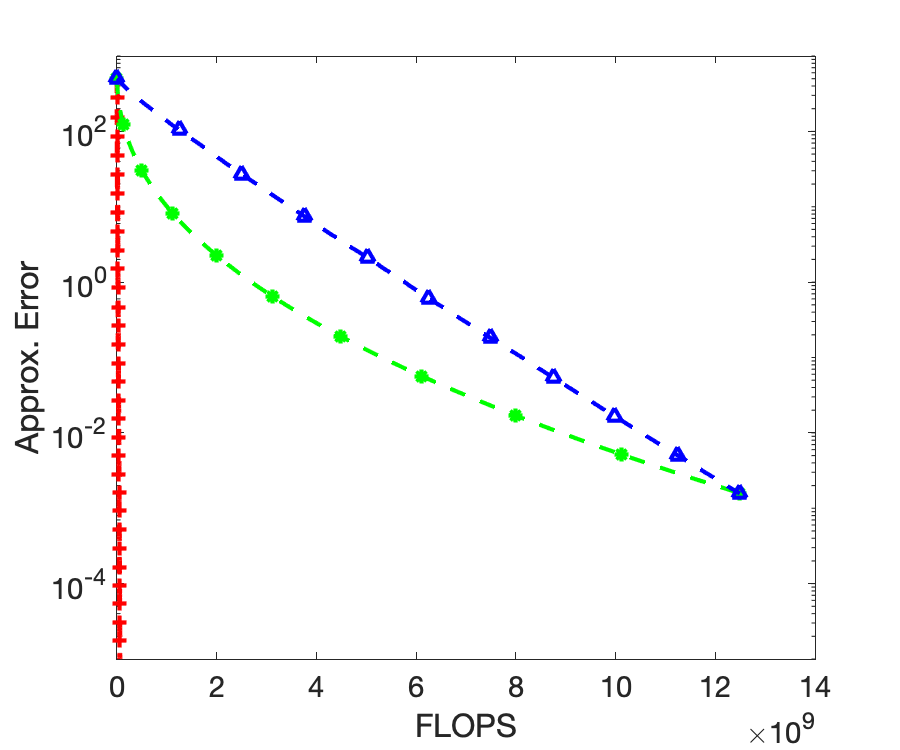}
    \includegraphics[width=.32\textwidth]{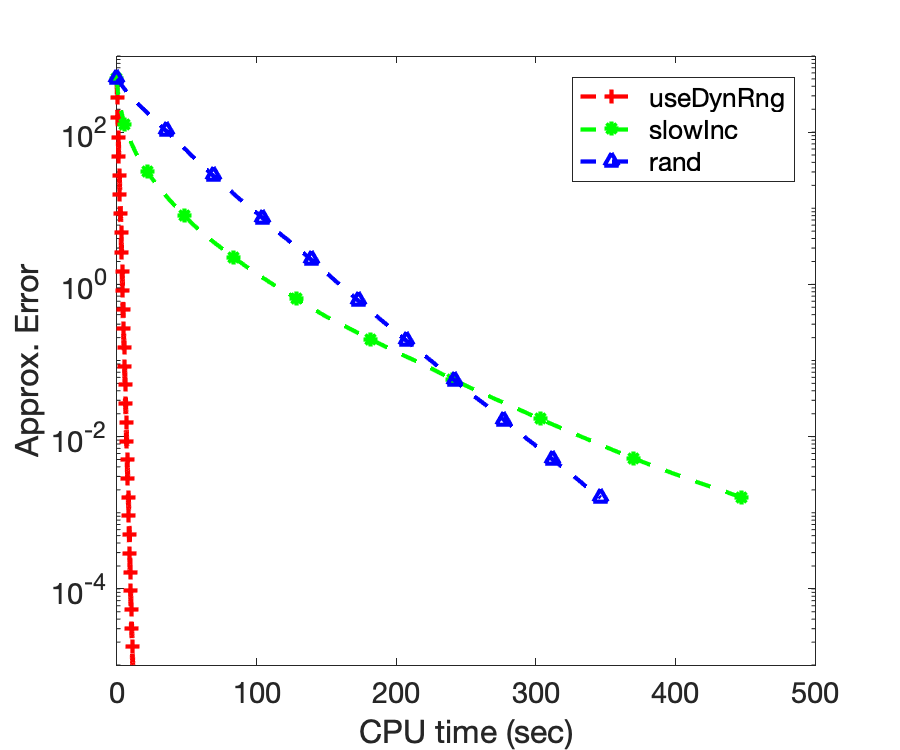}
\caption{Comparison of SKM for various choices of dynamically selected $\beta_k$ values on linear system with entries of $A$ drawn from Unif$([0,1])$. }
    \label{fig:dynamic_unif}
\end{figure}

Figure~\ref{fig:gaussbound} employs the upper bound on the dynamic range derived in Proposition~\ref{prop:gammabound} to approximate an upper bound for the error of SKM iterates when $\beta = 100$. Here, we compare the empirical performance of SKM with its previous known upper bound using the contraction term $1 - \frac{\sigma^2_{min}}{m}$ and $1 - \frac{\log(\beta)\sigma^2_{min}}{m}$. Note that we drop the factor of $n$ apparent in Proposition~\ref{prop:gammabound} as we suspect it to be an artifact of the proof technique used and conjecture that the true upper bound of the dynamic range is actually $\mathcal{O}(\beta / \log(\beta))$.

\begin{figure}
    \centering
    \includegraphics[width=.4\textwidth]{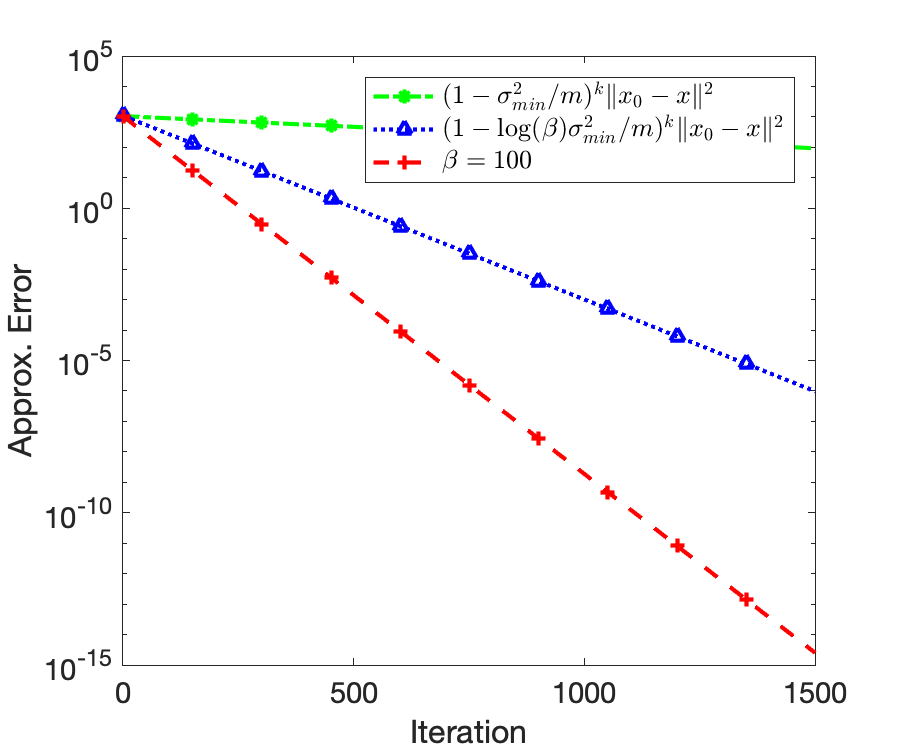}
\caption{Comparison of SKM with $\beta = 100$ with previous known theoretical upper bounds and our upper bound with conjectured Gaussian system dynamic range $\gamma_k$.}
    \label{fig:gaussbound}
\end{figure}

In both the Gaussian and Uniform synthetic experiments, the row norms of $A$ are of similar magnitude on average and thus choosing $\beta_k$ rows of the measurement matrix uniformly at random will behave similarly to the theoretically imposed probability distribution introduced in~\eqref{eq:probdist}. In the next experiment, we consider a setting where the entries of the measurement matrix are $a_{ij} \sim \mathcal{N}(0, i/\sqrt{n})$ so that for each row, $\mathbb{E}\|\ve a_i \|^2 = i$. Since~\eqref{eq:probdist} is computationally impractical to implement, we will continue to select rows of $A$ uniformly at random without replacement to evaluate the performance of SKM for various choices of $\beta$. The results of this experiment are presented in Figure~\ref{fig:mixedgauss}. As in the previous synthetically generated experiments, $m = 50000$, $n = 500$, and the underlying signal $x$ is a standard Gaussian random vector. Despite not sampling rows as imposed by \eqref{eq:probdist}, we see that SKM still converges with rates similar to those in Figure \ref{fig:gauss} and $\beta=100$ outperforms the others with respect to CPU time. 

\begin{figure}
    \centering
    \includegraphics[width=.32\textwidth]{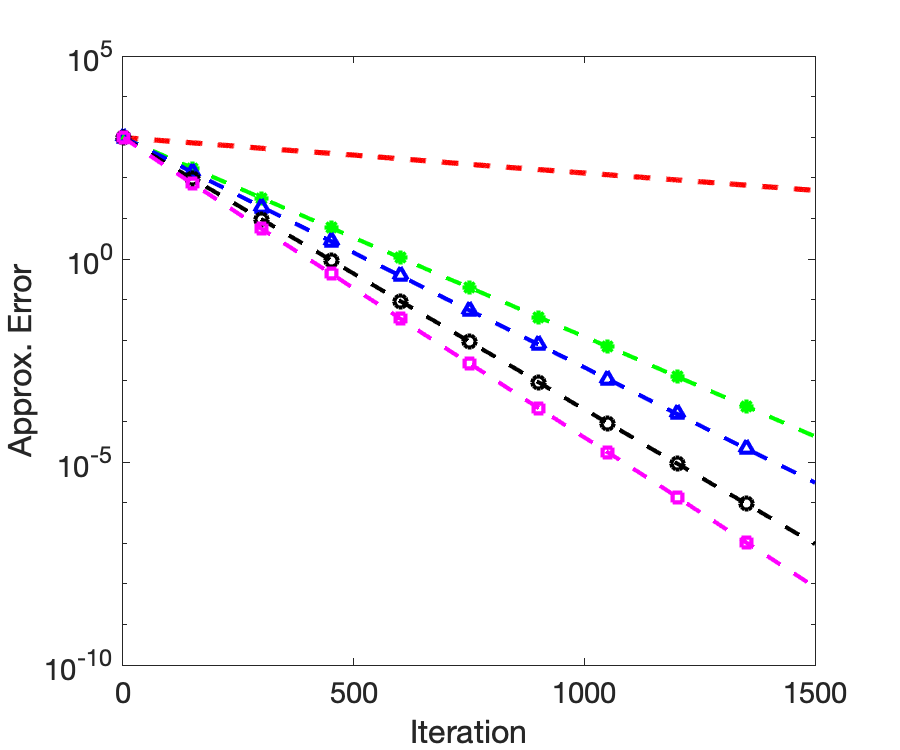}
     \includegraphics[width=.32\textwidth]{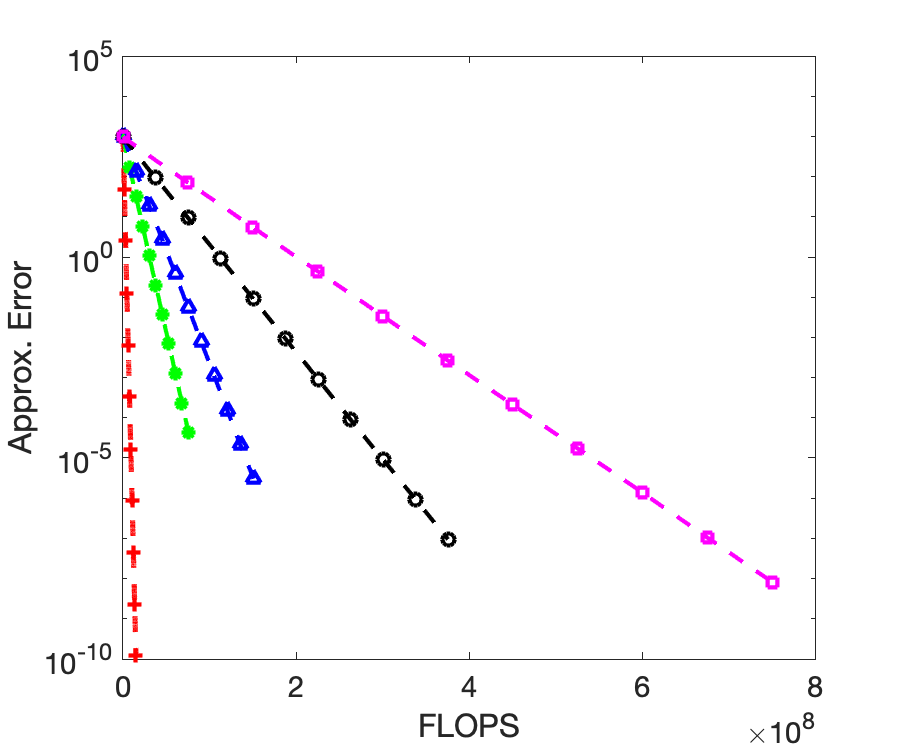}
      \includegraphics[width=.32\textwidth]{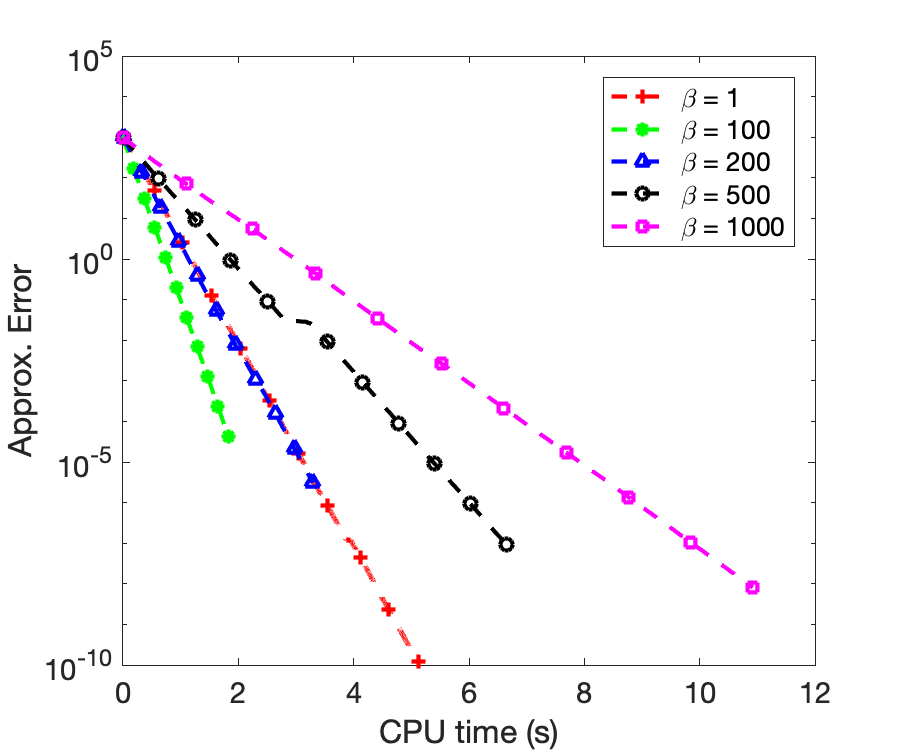}
\caption{Comparison of SKM for various choices of fixed $\beta$ values on linear system with entries $a_{ij}$ drawn from $\mathcal{N}(0, i/\sqrt{n})$. (left) Iteration vs Approximation Error with dashed lines representing average empirical performance of SKM and solid lines representing theoretical upper bounds for SKM. (middle) FLOPS vs Approximation Error. (right) CPU time vs Approximation Error}
    \label{fig:mixedgauss}
\end{figure}

Next we move on to evaluate the performance of SKM on incidence matrices of graphs. We start with the the AC systems discussed in Section~\ref{sec:avgcons}. Here, the graph $\mathcal{G}$ is a complete graph $K_{100}$ with corresponding incidence matrix $Q \in \{-1,0,1 \}^{4950 \times 100}$. The unknown underlying vector $\ve{x} \in \mathbb{R}^{100}$ is $\ve{x} = \hat{\mu}\ve{1}_{100}$ where $\ve{1}_{100}$ is a $100$-dimensional vector of ones and $\hat{\mu}$ is the empirical average of 100 random draws from a standard normal distribution. The results of this experiment are provided in Figure \ref{fig:ac_gauss}.

\begin{figure}
    \centering
    \includegraphics[width=.32\textwidth]{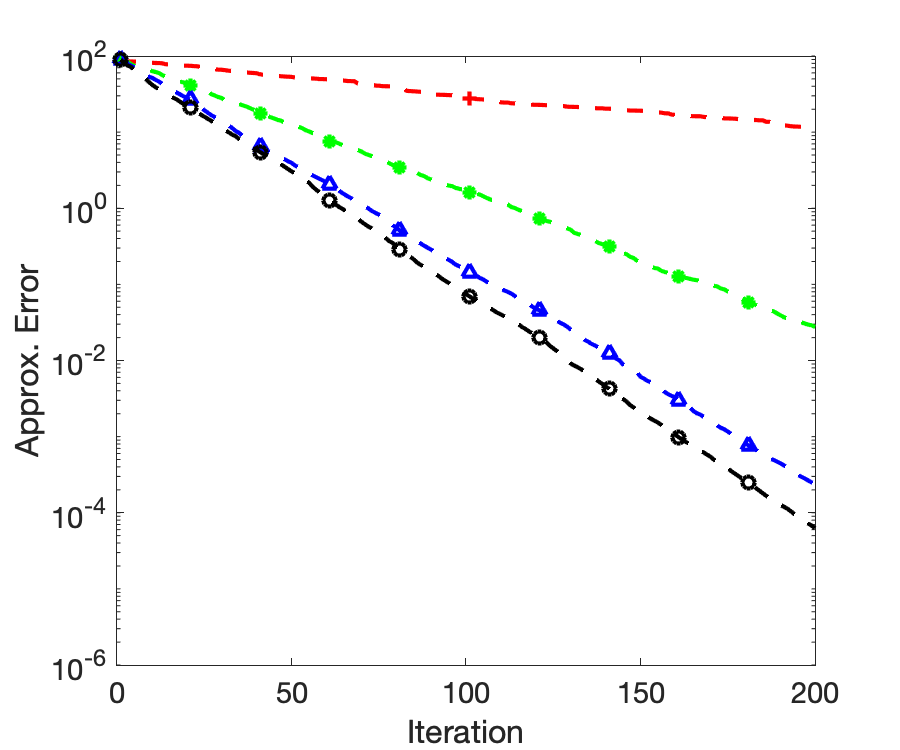}
     \includegraphics[width=.32\textwidth]{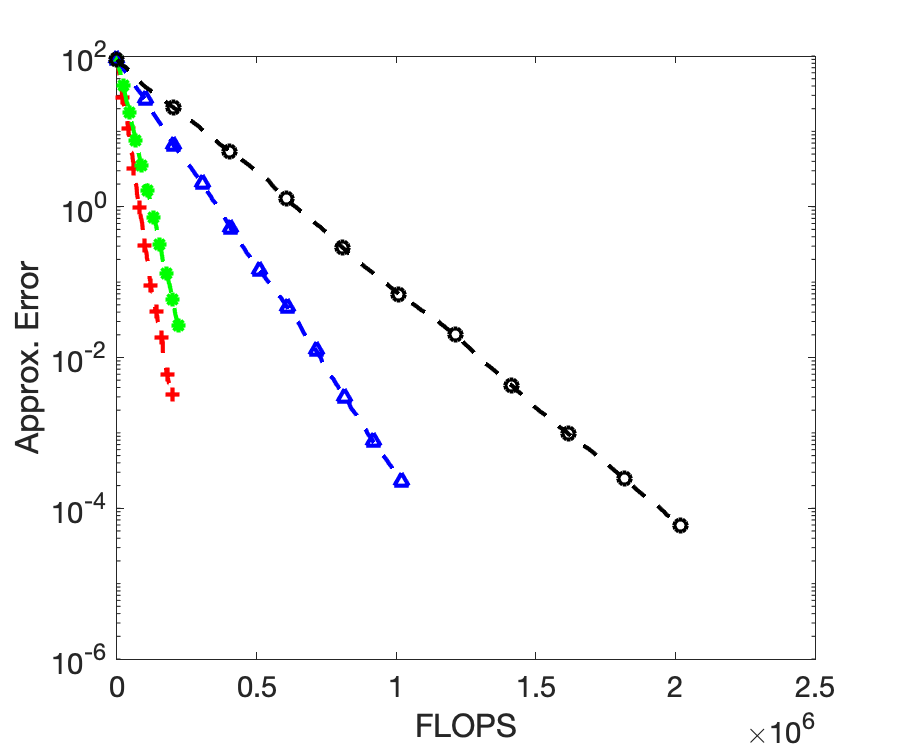}
      \includegraphics[width=.32\textwidth]{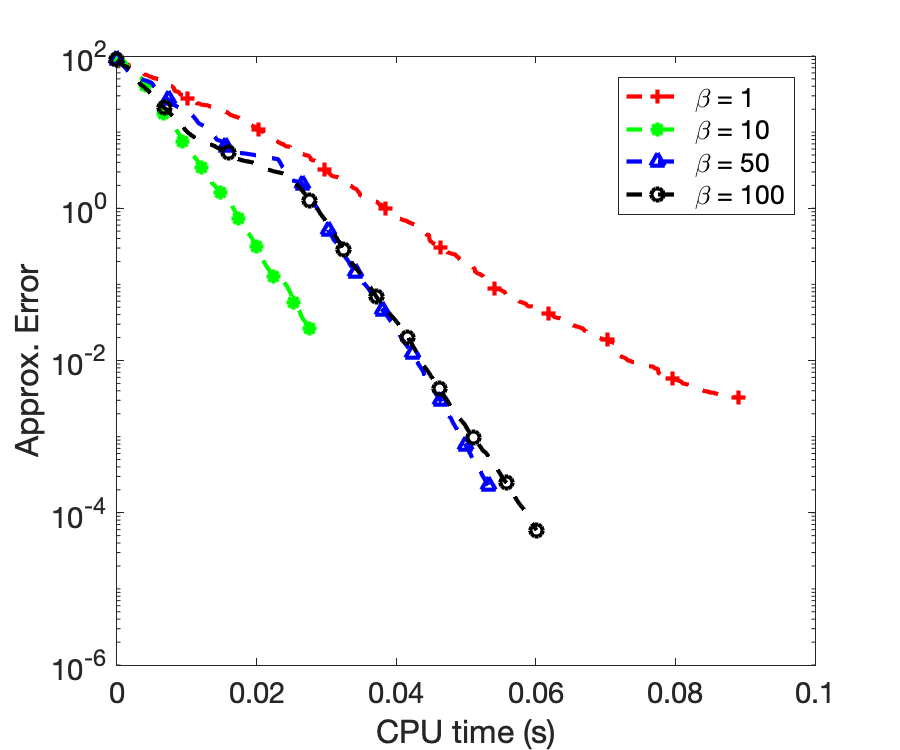}
\caption{Comparison of SKM for various choices of fixed $\beta$ on AC systems. (left) Iteration vs Approximation Error with dashed lines representing average empirical performance of SKM. (middle) FLOPS vs Approximation Error. (right) CPU time vs Approximation Error}
    \label{fig:ac_gauss}
\end{figure}

Figure~\ref{fig:ac_pl} demonstrates the performance of SKM on a graph $\mathcal{G}$ which reflects a scale-free network, i.e., a graph whose degree distribution follows the power law. To create the graph, we employ the implementation of the Barab\'{a}si-Albert (BA) model~\cite{albert2002statistical} with an initial graph of five vertices and ending with a graph of 300 vertices ~\cite{baa}. For more details on scale-free networks, see~\cite{albert2002statistical}. In Figure~\ref{fig:ac_pl} we again observe exponential convergence in the mean approximation error and optimal performance with respect to CPU time when $\beta = 10$.

\begin{figure}
    \centering
    \includegraphics[width=.32\textwidth]{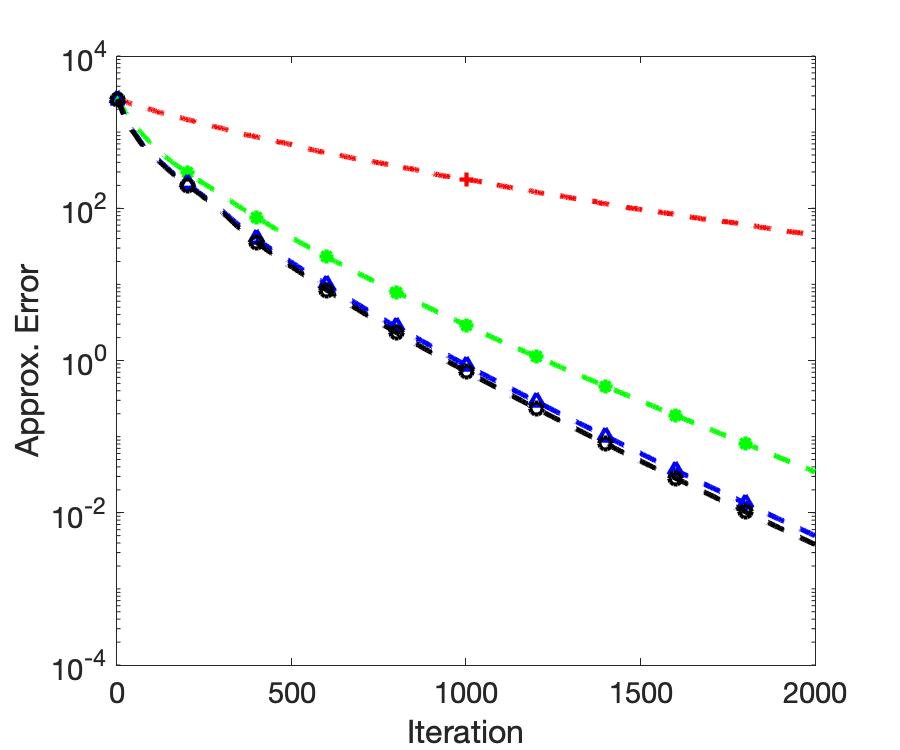}
     \includegraphics[width=.32\textwidth]{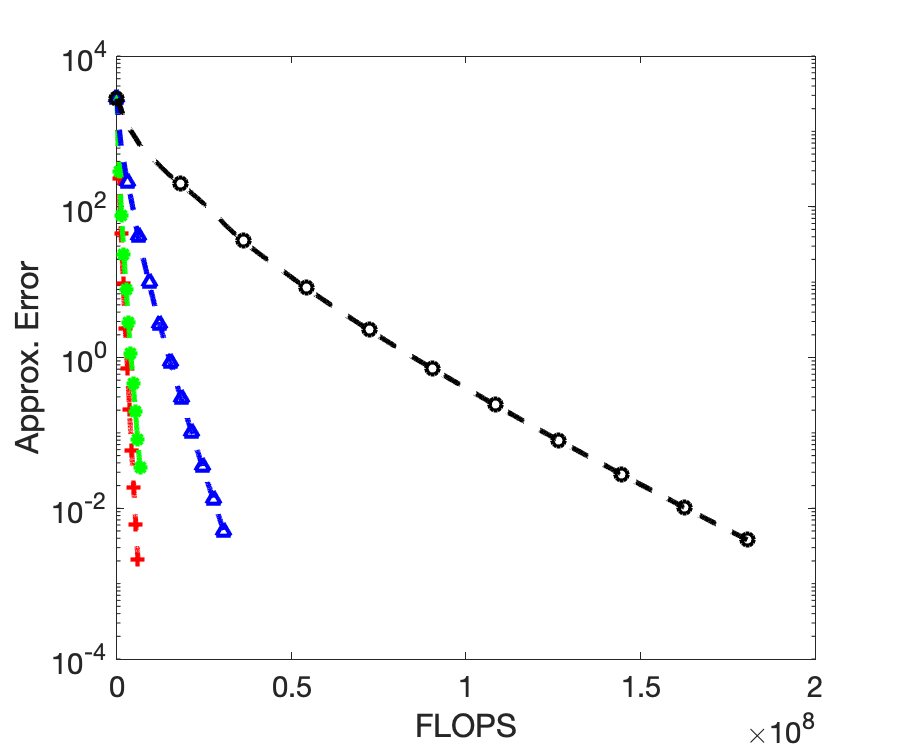}
      \includegraphics[width=.32\textwidth]{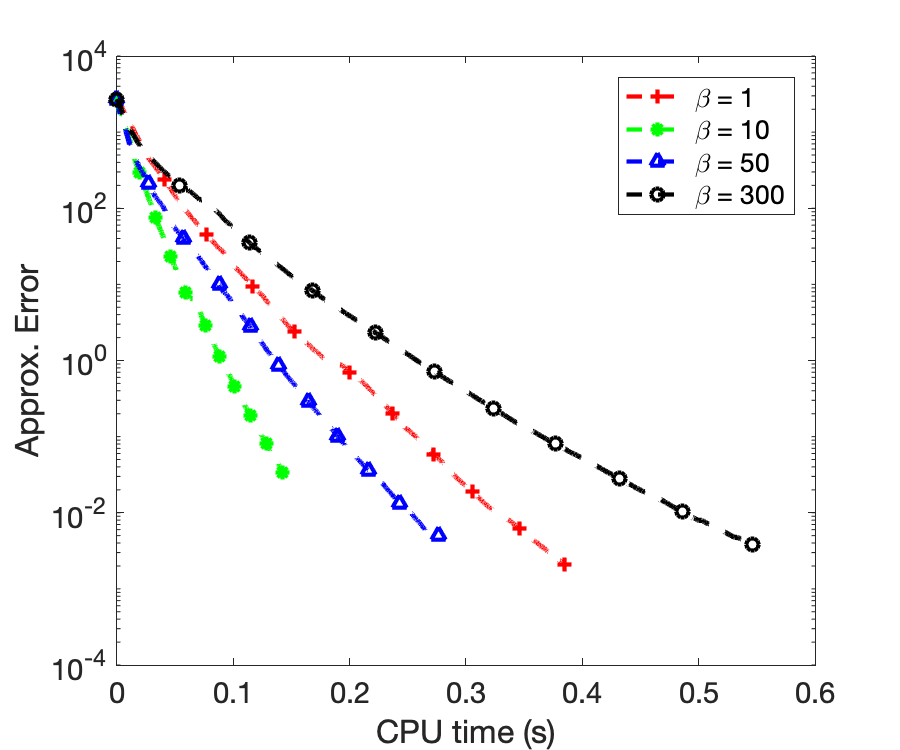}
\caption{Comparison of SKM for various choices of fixed $\beta$ on an incidence matrix of a scale-free network. (left) Iteration vs Approximation Error with dashed lines representing average empirical performance of SKM. (middle) FLOPS vs Approximation Error. (right) CPU time vs Approximation Error}
    \label{fig:ac_pl}
\end{figure}

Our last experiment demonstrates the performance of SKM with varying choices of $\beta$ on real world data from the SuiteSparse Matrix Market~\cite{davis2011university}. For these experiments, we employ the `Maragal\_4', `well1850', and `ash958' matrices which are of dimension $1964 \times 1034$, $1850 \times 712$, and $958 \times 292$ respectively. It is useful to note that `Maragal\_4' and  `well1850' have low rank data structures. These matrices are used as the measurement matrix $A$ and the underlying signal is arbitrarily chosen to be in the range of $A^T$. Here, we allow SKM to run for a maximum of $10^6$ iterations or terminate once the approximation error has reached the allotted error tolerance of $10^{-6}$. The results of using $\beta \in \{1, 10, 50 \}$ are presented in Table~\ref{tab1} and Table~\ref{tab2}. We also included the performance of Conjugate Gradient Least Squares (GCLS)~\cite{bjorck1996numerical, cgls} for comparison. Note that we do not claim to have optimized either implementation. 

Table~\ref{tab1} and Table~\ref{tab2} demonstrates that for a fixed number of iterations, increasing $\beta$ results in a lower approximation error. It also highlights the trade off between the subset size, the FLOP cost, and the CPU time. As we increase $\beta$, in general, both the FLOP and CPU time increase as well. One interesting observation is that for the `ash958' experiment, the optimal choice of subset size is $\beta=10$ with respect to CPU time. This is further motivation for future work in optimal $\beta$ selection. Finally, as expected, CGLS outperforms the three choices of $\beta$. However, SKM can be more naturally implemented in distributed computing settings. We leave that direction as an avenue for future work as well.

\begin{table}[]
\centering
\begin{tabular}{l|l|l|l|l|l|l|}
\cline{2-7}
 & \multicolumn{3}{c|}{\textbf{SKM $\beta = 1$}} & \multicolumn{3}{c|}{\textbf{SKM $\beta = 10$}} \\ \cline{2-7} 
 & Error        & FLOPS        & CPU time        & Error          & FLOPS         & CPU time      \\ \hline
\multicolumn{1}{|l|}{Maragal\_4} &    2.45          &    2.07e+09          &       176.25          & 0.164          & 1.13e+10      & 460.88        \\ \hline
\multicolumn{1}{|l|}{well1850}   &    7.67          &     1.42e+09         &     112.75            & 0.064          & 7.83e+09      & 238.59        \\ \hline
\multicolumn{1}{|l|}{ash958}     &     1.06e-06         &   4.91e+06           &       0.323          & 1.00e-06       & 5.33e+06      & 0.126         \\ \hline
\end{tabular}
\caption{Performance of SKM using $\beta=1$ and $\beta=10$ on real world data from SuiteSparse Matrix Market.}
\label{tab1}
\end{table}

\begin{table}[]
\centering
\begin{tabular}{l|l|l|l|l|l|l|}
\cline{2-7}
\textbf{}                        & \multicolumn{3}{c|}{\textbf{SKM $\beta=50$}} & \multicolumn{3}{c|}{\textbf{CGLS}} \\ \cline{2-7} 
                                 & Error         & FLOPS         & CPU time     & Error      & FLOPS     & CPU time  \\ \hline
\multicolumn{1}{|l|}{Maragal\_4} & 5.01e-02       & 5.27e+10      & 1593.4       & 3.97e-02    & 2.97e+07  & 5.41      \\ \hline
\multicolumn{1}{|l|}{well1850}   & 2.49e-03       & 3.63e+10      & 1161         & 1.01e-06   & 4.06e+06  & 1.07      \\ \hline
\multicolumn{1}{|l|}{ash958}     & 1.01e-06      & 1.44e+07      & 0.426        & 1.88e-06   & 41158     & 6.66e-3   \\ \hline
\end{tabular}
\caption{Performance of SKM using $\beta=50$ and CGLS on real world data from SuiteSparse Matrix Market.}
\label{tab2}
\end{table}

\section{Conclusion}
This work unifies the spectrum between the randomized Kaczmarz and a greedy variant of the Kaczmarz (Motzkin's Method) algorithm by improving the convergence bound of SKM, a hybrid randomized-greedy algorithm. We show that the behavior of SKM depends on the sample parameter $\beta_k$ and the dynamic range of the linear system. This result improves upon previous work showing only the linear convergence of SKM. In presenting an improved convergence bound for SKM that highlights the impact of the sub-sample size $\beta_k$, we have opened up new and exciting avenues for SKM-type algorithms. Future directions of this work include finding optimal sample sizes for different types of linear systems and designing adaptive sample size selection schemes.

\section*{Acknowledgements}
The authors would like to thank the manuscript referees for their thoughtful and detailed comments which significantly improved earlier versions of this work. The authors also thank Jacob Moorman, Liza Rebrova, Hanbaek Lyu, Deanna Needell, Jes{\'u}s A. De Loera, and Roman Vershynin for useful conversations and suggestions.

\bibliographystyle{abbrv}
\bibliography{basebib,SKMbib}

\end{document}